\documentclass[12pt]{amsart}
\usepackage{amsmath,amssymb,amsthm}
\usepackage{fancybox,ascmac}
\usepackage[dvipdfmx]{graphicx}

\usepackage{txfonts}
\usepackage{graphicx}
\usepackage{wrapfig}
\usepackage{pifont}
\usepackage{color}

\usepackage[left=30mm,right=30mm,top=30mm,bottom=30mm]{geometry}

\newtheorem{dfn}{Definition}[section]
\newtheorem{thm}[dfn]{Theorem}
\newtheorem{lem}[dfn]{Lemma}
\newtheorem{cor}[dfn]{Corollary}
\newtheorem{rem}[dfn]{Remark}

\newtheorem{prop}[dfn]{Proposition}\makeatletter

\makeatletter
\renewenvironment{proof}[1][\proofname]{\par
  \normalfont
  \topsep6\p@\@plus6\p@ \trivlist
  \item[\hskip\labelsep{\bfseries #1}\@addpunct{\bfseries.}]\ignorespaces
}{%
  \endtrivlist
}
\renewcommand{\proofname}{Proof.}
\makeatother

\numberwithin{equation}{section}

\begin{document}

\title{Exponents for the number of pairs of $\alpha$-favorite points of a simple random walk in ${\mathbb Z}^2$}

%    Remove any unused author tags.

%    author one information
\author{Izumi Okada}
%\address{Izumi Okada, Tokyo Institute of Technology}
%\curraddr{Ookayama 2-12-1, Meguro-ku, Tokyo, 152-8550, Japan.}
%\email{okada.i.aa@m.titech.ac.jp}
\thanks{Address: Tokyo University of Science
Noda City, Chiba 278-8510, Japan.\\
           \quad E-mail: izumiokada1205@gmail.com\\
           %\quad Research partially supported by JSPS Research Fellowships.\\
           \quad Short title: Favorite points of a simple random walk}
%    author two information
%\author{}
%\address{}
%\curraddr{}
%\email{}
%\thanks{}

%\subjclass[2010]{Primary }

%\keywords{}

\date{}

\dedicatory{}

\begin{abstract}
We investigate a problem suggested by Dembo, Peres, Rosen, and Zeitouni, which states that the growth exponent of favorite points associated with a simple random walk in $\mathbb Z^2$ coincides, on average and almost surely, with those of late points and high points associated with the discrete Gaussian free field.
\end{abstract}

\maketitle

\section{Introduction, known results, and unsolved problems}
\quad In this paper, we study some properties of the local time of a simple random walk,  %and some specific sites of its  range. 
thereby positively resolving a conjecture suggested in \cite{Dembo2}. 
Although this topic has been studied extensively, %many questions on this topic have still been extensively studied, 
numerous open problems still exist. 
%Let $(S_k)_{k=1}^{\infty}$ be a simple random walk on the $2$-dimensional square lattice. 
Approximately 60 years ago, 
Erd\H{o}s and Taylor \cite{er} proposed a problem concerning a simple random walk in $\mathbb{Z}^2$: 
How many times does a simple random walk  
$(S_k)_{k=1}^{\infty}$ revisit the most frequently visited site (up to a specific time)? 
The points that are most often visited in a simple random walk 
are called {\it{favorite points}}. 
In \cite{er}, it is shown that, for a simple random walk in ${\mathbb{Z}^2}$,
\begin{align*}
 \frac{1}{4\pi}
\le \liminf_{n \to \infty} \frac{\max_{x\in {\mathbb Z}^2}K(n,x)}{(\log n)^2}
\le \limsup_{n \to \infty} \frac{\max_{x\in {\mathbb Z}^2}K(n,x)}{(\log n)^2}
\le \frac{1}{\pi} \quad \text{a.s.,} 
\end{align*}
where $K(n,x)$ is the number of times the simple random walk visits $x$ up to time $n$, 
that is, $K(n,x)=\sum_{i=1}^n1_{\{S_i=x\}}$.  
%and  $T_D:=\inf \{m\ge1: S_m\in D\}$ for $D\subset \mathbb{Z}^2$. 
%They computed certain upper and lower bounds for the number of the favorite points and also gave a conjecture on the exact asymptotic form of it. 
%Now, we introduce the known results of favorite points. 
%Originally, Erd\H{o}s and Taylor \cite{er} observed where a local time of a simple random walk in two dimensions is high and proposed the conjecture. 
Erd\H{o}s and Taylor conjectured the existence and value of this limit. 
This problem was solved by Dembo et al. \cite{Dembo} approximately 20 years ago.  
In fact, they showed that, for a simple random walk in ${\mathbb{Z}^2}$, 
\begin{align*}
\lim_{n \to \infty} \frac{\max_{x\in {\mathbb Z}^2}K(\tau_n,x)}{4(\log n)^2}
=\lim_{n \to \infty} \frac{\max_{x\in {\mathbb Z}^2}K(n,x)}{(\log n)^2}
= \frac{1}{\pi} \quad \text{a.s,} 
\end{align*}
where $\tau_n$ is the first time that the simple random walk exits a disc around the origin of radius $n$.  
%where $K(n,x)$ is the number of times the simple random walk visits $x$ up to time $n$, 
%that is, $K(n,x)=\sum_{i=0}^n1_{\{S_i=x\}}$ and  
%$T_D:=\inf \{m\ge1: S_m\in D\}$ for $D\subset \mathbb{Z}^2$. 
In addition, Dembo et al. \cite{Dembo} 
showed that
\begin{align}\label{fav1}
\frac{\log |\Psi_n(\alpha)|}{\log n}\to 2(1- \alpha) 
\quad \text{ as } n \to  \infty \quad \text{ a.s.,}
\end{align} 
where $\Psi_n (\alpha)$ is the set of $\alpha$-favorite points defined as
\begin{align*}
 \Psi_n(\alpha):=\bigg\{x\in D(0,n): K(\tau_n,x)\ge \bigg\lceil \frac{4\alpha}{\pi} (\log n)^2 \bigg\rceil \bigg\}
 \end{align*}
for $0<\alpha<1$, where $\lceil a \rceil$ is the smallest integer $n$ with $n \ge a$, 
 $D(x,r):=\{y\in \mathbb{Z}^2: d(x,y)\le r\}$,  and $d$ is the Euclidean distance.    
In other words, $\Psi_n (\alpha)$ is the set of points whose local time is  an $\alpha$-fraction of the favorite point. 
They derived these results using an innovative application of the second moment method.  
(Note that Rosen \cite{rosen} provided another proof of \cite{Dembo}.) 
%Around ten years ago, Dembo et al. \cite{Dembo} positively resolved their conjecture and 
In addition, Dembo et al. suggested additional open problems concerning favorite points. 

%There are a few known results regarding geometric structures of the favorite point. 
Recently, research on the geometric structure of the favorite points of a simple random walk in $\mathbb{Z}^d$ has progressed significantly. 
For example, in \cite{okada1}, it was shown that the favorite point of a simple random walk in $\mathbb{Z}^d$ with $d\ge 2$ almost surely does not appear in the inner boundary of the random walk range after a sufficiently large time.  
Lifshits and Shi \cite{shi2} verified that the favorite point of a one-dimensional simple random walk tends to be far from the origin, 
and the number of times  a simple random walk revisits the favorite point up to the cover time on a two-dimensional torus has been estimated in \cite{abe}. 
Several open problems concerning favorite points were raised by Erd\H{o}s and R\'ev\'esz \cite{er2, er3} and Shi and T\'oth \cite{shi}, but almost no  solutions have been derived for multi-dimensional walks.

In this paper, we focus on $\alpha$-favorite points. 
%We recall the set of  $\alpha$-favorite points $\Psi_n(\alpha)$ for $0<\alpha<1$ as in \cite{Dembo2} (see (\ref{ff1+}) in the next section). 
We obtain asymptotic formulas for 
\begin{align}\label{sss1}
|\{ (x,x')\in \Psi_n(\alpha)^2:d(x,x')\le n^{\beta} \}|
\end{align}
as $n\to \infty$, where $0< \alpha,\beta <1$ and $|A|$ denotes the cardinality of $A$. 
%In \cite{Dembo2}, they conjectured (\ref{m1}). 
Asymptotic formulas such as this can help to determine the joint distribution of the long-time behavior of local time scales. 
%The present study of the $\alpha$-favorite points is motivated by the desire to understand the compound distribution of the long‐time behavior of local times. 
In addition, we want to compare the long-time behavior of the local time of a random walk to 
the values of the discrete Gaussian free field (DGFF). 
This is inspired by the results presented in \cite{ol} and \cite{Dembo2}, which show 
that the asymptotic behavior of (\ref{sss1}) is similar to that of the number of $\alpha$-high points  of the DGFF in $V_n:=\{1,\ldots,n\}^2$ with zero boundary conditions and, for a simple random walk, the number of $\alpha$-late points in $\mathbb{Z}^2_n(=\mathbb{Z}^2/n\mathbb{Z}^2)$  for $0<\alpha<1$ ($\alpha$-high points and $\alpha$-late points are defined in Section \ref{fre}). 
%Here $\alpha$-high points mean the site where DGFF takes high values and 
%$\alpha$-late points mean the site where the hitting time of the random walk is large. 
%There is a similarity concerning $\alpha$-high points, late points and favorite points as we mention in Section $2$. 
%Daviaud \cite{ol} estimated the number of $\alpha$-high points of the DGFF in the same forms as (\ref{m1}) and (\ref{m2}). 
%For $\alpha$-late points in $\mathbb{Z}^2_n$, in \cite{Dembo2} they obtained  the result corresponding to (\ref{m1}) 
%and Brummelhuis and Hilhorst  \cite{ho} that  corresponding to (\ref{m2}). 
%In each case, the exponent coincides with that of  (\ref{m1}) or (\ref{m2}). 
%Then, these results give random configuration of $\alpha$-high points of DGFF and $\alpha$-late points. 
%In \cite{Eisen}, they obtained some relations between the local time and the DGFF, and according to them together with the results mentioned right above we have expected our results to hold true.  

%%%%%%%%%%%%%%%%%%%%%%%%%%%%%%%%%%%%%%%%%%%%%%%%%%%%%%%%%%%%%%%%%%%%%%%%%%%%%%%%%%%%%%%%%%%%%%%%%%%%%%%%%%%%%%%%

\section{Framework and principal results}\label{fre}
%Let $d$ be the Euclidean distance. 
%Let $D(x,r):=\{y\in \mathbb{Z}^2: d(x,y)\le r\}$ and for any $G\subset \mathbb{Z}^2$, 
%$\partial G:=\{y\in \mathbb{Z}^2\setminus G: d(x,y)=1 \text{ for some }x\in G \}$. 
%For $x\in \mathbb{Z}^2$, we sometimes omit $\{\}$ for the singleton set $\{x\}$. 
%Let $(S_k)_{k=1}^{\infty}$ be a simple random walk on the $2$-dimensional square lattice.  
%the sigma field ${\cal F}(n)$ be $\sigma(S_k: k\le n)$ 
%and ${\cal F}(T)$ be the sigma field $\sigma(S_k: k\le T)$
%for any stopping time $T$. 
%Let $K(n,x)$ be the number of times the simple random walk visits $x$ up to time $n$, that is, $K(n,x)=\sum_{i=0}^n1_{\{S_i=x\}}$. 
%For any $D\subset \mathbb{Z}^2$, let $T_D:=\inf \{m\ge1: S_m\in D\}$. 
%For simplicity, we write $T_{x_1,\ldots,x_j}$ for $T_{\{x_1,\ldots,x_j\}}$. 
%Let $\tau_n:=\inf \{m\ge 0: S_m\in \partial D(0,n)\}$. 
%$\lceil a \rceil$ denotes the smallest integer $n$ with $n \ge a$. 
%We will give  the problem corresponding to \cite{Dembo2} by changing $\alpha$-late points  into $\alpha$-favorite points  as follows. 
For $0<\alpha, \beta<1$ and $\gamma \ge0$,  let
\begin{align*}
F_{h,\beta}(\gamma):&=\gamma^2(1-\beta)+\frac{h}{\beta}(1-\gamma(1-\beta))^2, \\
\Gamma_{\alpha,\beta}:&=\{\gamma\ge0: 2-2\beta- 2\alpha F_{0,\beta}(\gamma)\ge 0\}\\
&=\{\gamma\ge0: \alpha \gamma^2\le 1\}.
\end{align*}
The principal results of this paper are as follows. 
\begin{thm}\label{h1}
For any $0< \alpha,\beta <1$, %with probability one
\begin{align*}
\lim_{n\to \infty}\frac{\log 
|\{ (x,x')\in \Psi_n(\alpha)^2:d(x,x')\le n^{\beta} \}|}{\log n}
=\rho_2(\alpha, \beta) \quad \text{a.s.},
\end{align*}
where
\begin{align*}
\rho_2(\alpha, \beta):=
2+2\beta -2\alpha \inf_{\gamma\in \Gamma_{\alpha,\beta}}F_{2,\beta}(\gamma)
=
\begin{cases}
 2+2\beta-\frac{4\alpha}{2-\beta}&(\beta\le 2(1-\sqrt{\alpha})),
\\
8(1-\sqrt{\alpha})-4(1-\sqrt{\alpha})^2/\beta&(\beta\ge 2(1-\sqrt{\alpha})).
\end{cases}
\end{align*}
\end{thm}

\begin{thm}\label{h2}
For any $0< \alpha,\beta <1$,
\begin{align*}
\lim_{n\to \infty}\frac{\log 
E[|\{ (x,x')\in \Psi_n(\alpha)^2:d(x,x')\le n^{\beta} \}|]}{\log n}
=\hat{\rho}_2(\alpha, \beta),
\end{align*}
where
\begin{align*}
\hat{\rho}_2(\alpha, \beta):=
\sup_{\beta'\le \beta}\sup_{\gamma \ge0}\{2+2\beta' -2\alpha F_{2,\beta'}(\gamma)\}
=
\begin{cases}
2+2\beta-\frac{4\alpha}{2-\beta}&(\beta\le 2-\sqrt{2\alpha}),
\\
 6-4\sqrt{2\alpha}&(\beta\ge 2-\sqrt{2\alpha}).
\end{cases}
\end{align*}
\end{thm}

\begin{rem}
As $\beta \to 0$, we have
\begin{align*}
\frac{\log E|\Psi_n(\alpha)|}{\log n}\to 2(1- \alpha) \quad \text{ as } n \to  \infty.
\end{align*}
This limit is the same as in (\ref{fav1}).
\end{rem}

%To state the relation between $\hat{\rho}_2(\alpha, \beta)$ and $\rho_2(\alpha, \beta)$,  let
%\begin{align*}
%F_{h,\beta}(\gamma):=\gamma^2(1-\beta)+\frac{h}{\beta}((1-\gamma(1-\beta)))^2 \quad
%\text{for }\gamma>0.
%\end{align*}
%We state the relation between $\hat{\rho}_2(\alpha, \beta)$ and $\rho_2(\alpha, \beta)$. 
We now present an intuitive explanation of the roles played by the variables $h$ and $\gamma$.  
We will estimate the probability that specific points are $\alpha$-favorite points conditioned by crossing numbers around these points.  
Roughly speaking, we will show that the probability can be written as a functional of $F_{h,\beta}(\gamma)$. 
$\gamma$ is a crossing number parameter, where $\gamma^2(1-\beta)$ corresponds to the probability that a simple random walk crosses a particular circle, and 
$h(1-\gamma(1-\beta))^2/\beta$ corresponds to the  probability that $h$-tuple points are $\alpha$-favorite points conditioned by their crossing numbers. 
%It can be easily checked that 
%\begin{align*}
%\rho_2(\alpha, \beta)=2+2\beta -2\alpha \inf_{\gamma\in \Gamma_{\alpha,\beta}}F_{2,\beta}(\gamma),
%\end{align*}
%where
%\begin{align*}
%\Gamma_{\alpha,\beta}&=\{\gamma\ge0: 2-2\beta- 2\alpha F_{0,\beta}(\gamma)\ge 0\}\\
%&=\{\gamma\ge0: \alpha \gamma^2\le 1\}.
%\end{align*}
%We easily verify that 
%\begin{align}\label{j1}
%\hat{\rho}_2(\alpha, \beta)
%=\sup_{\beta'\le \beta}\sup_{\gamma \ge0}\{2+2\beta' -2\alpha F_{2,\beta'}(\gamma)\}.
%\end{align}
%Hence, the difference between $\rho_2(\alpha, \beta)$ and $\hat{\rho}_2(\alpha, \beta)$ 
%is the domain of $\gamma$ in the supremum in (\ref{j1}). 
We will explain why the domain of a function $F_{2,\beta}(\gamma)$ to have  $\rho_2(\alpha, \beta)$ has to be restricted  in $\Gamma_{\alpha,\beta}$ 
after proving Proposition \ref{propq2} in Section \ref{low Th h1}. 
%This restriction corresponds to Lemma \ref{r1+}  
%and hence, that of $F_{2,\beta}(\gamma)$ restricted in $\Gamma_{\alpha,\beta}$. 
%Therefore, we obtain different quantities for the exponent in probability  from that  in average.   
See Section $1.2$ in \cite{Dembo*} for more details concerning this difference. 

We now present the corresponding known results regarding special points of the DGFF in $V_n$ and a simple random walk in 
$\mathbb{Z}^2_n$. %(:=\mathbb{Z}^2/n\mathbb{Z}^2)$. 
While not directly connected with $\alpha$-favorite points, the results are important. 
First, we consider $\alpha$-high points of the DGFF in $V_n$. 
Bolthausen et al. \cite{Bol} showed that in probability
\begin{align*}
\lim_{n \to \infty}\frac{ \max_{x \in V_n }\phi_n(x)}{\log n}=2\sqrt{\frac{2}{\pi}},
\end{align*}
where $\{\phi_n(x)\}_{x \in V_n}$ is the DGFF with zero boundary conditions. 
In what follows, for $0<\alpha<1$, we define the set of $\alpha$-high points of the DGFF as
\begin{align*}
\mathcal{V}_n(\alpha):=\bigg\{ x\in V_n :
\frac{\phi_n(x)^2}{2} \ge \bigg\lceil  \frac{4\alpha}{\pi} (\log n)^2  \bigg\rceil  \bigg\}.
\end{align*}
Daviaud \cite{ol} showed that if $\Psi_n(\alpha)$ in Theorems \ref{h1} and \ref{h2} is replaced by $\mathcal{V}_n(\alpha)$, 
the exponent is the same as that of Theorems \ref{h1} and \ref{h2}.  
In addition, there are some interesting results concerning the relation between the local time of a random walk and the DGFF. 
Eisenbaum et al. \cite{Eisen} reported a powerful equivalence law called the ``generalized second Ray-Knight theorem". 
The relations they obtained between the local time and the DGFF 
are vital in deriving the principal results presented in this paper.   %(see Remark 1.1). 
Ding et al. \cite{Ding2} reported a strong correlation between the expected maximum of the DGFF and the expected cover time (see also \cite{Ding1}). 

Next, we define the $\alpha$-late points of a simple random walk in $\mathbb{Z}^2_n$. 
For a simple random walk in ${\mathbb{Z}^2_n}$, Dembo et al. \cite{Dembo3} showed that,
in probability,
\begin{align*}
\lim_{n \to \infty} \frac{\max_{x\in {\mathbb Z}^2_n}T_x }{(n\log n)^2}= \frac{4}{\pi}, 
\end{align*}
where $T_x$ is the hitting time of $x$ for a simple random walk in $\mathbb{Z}^2_n$.  
Subsequently, for $0<\alpha<1$, we define the set of $\alpha$-late points in $\mathbb{Z}^2_n$ as 
\begin{align*}
\mathcal{L}_n(\alpha):=\bigg\{ x\in {\mathbb{Z}^2_n} :
\frac{T_x}{(n\log n)^2}\ge \frac{4\alpha}{\pi} \bigg\}.
\end{align*}
%In \cite{Dembo2}, Dembo, Peres, Rosen and Zeitouni showed  
%if $\Psi_n(\alpha)$ in (\ref{h1}) is changed into ${\cal L}_n(\alpha)$, the exponent is same as Theorem \ref{h1}.   
Brummelhuis and Hilhorst \cite{ho} obtained a result for $\mathcal{L}_n(\alpha)$ corresponding to Theorem \ref{h2}, 
and verified the same formula with $\mathcal{L}_n(\alpha)$ in place of $\Psi_n(\alpha)$. 
Similarly,  in \cite{Dembo2}, a result corresponding to Theorem \ref{h1} was obtained. 

Next, we explain our proof. 
%We focus on the upper bound  and the lower bound in Theorem \ref{h1}. 
The upper bound in Theorem \ref{h1} can be proved using the method developed in \cite{Dembo1, Dembo2}. 
In \cite{Dembo1}, an estimate is derived for the probability that two points  in a disk are covered by a random walk under a particular condition on the crossing number between two circles (see Figure 1). 
\begin{figure}[h]
\centering
\includegraphics[width=110mm]{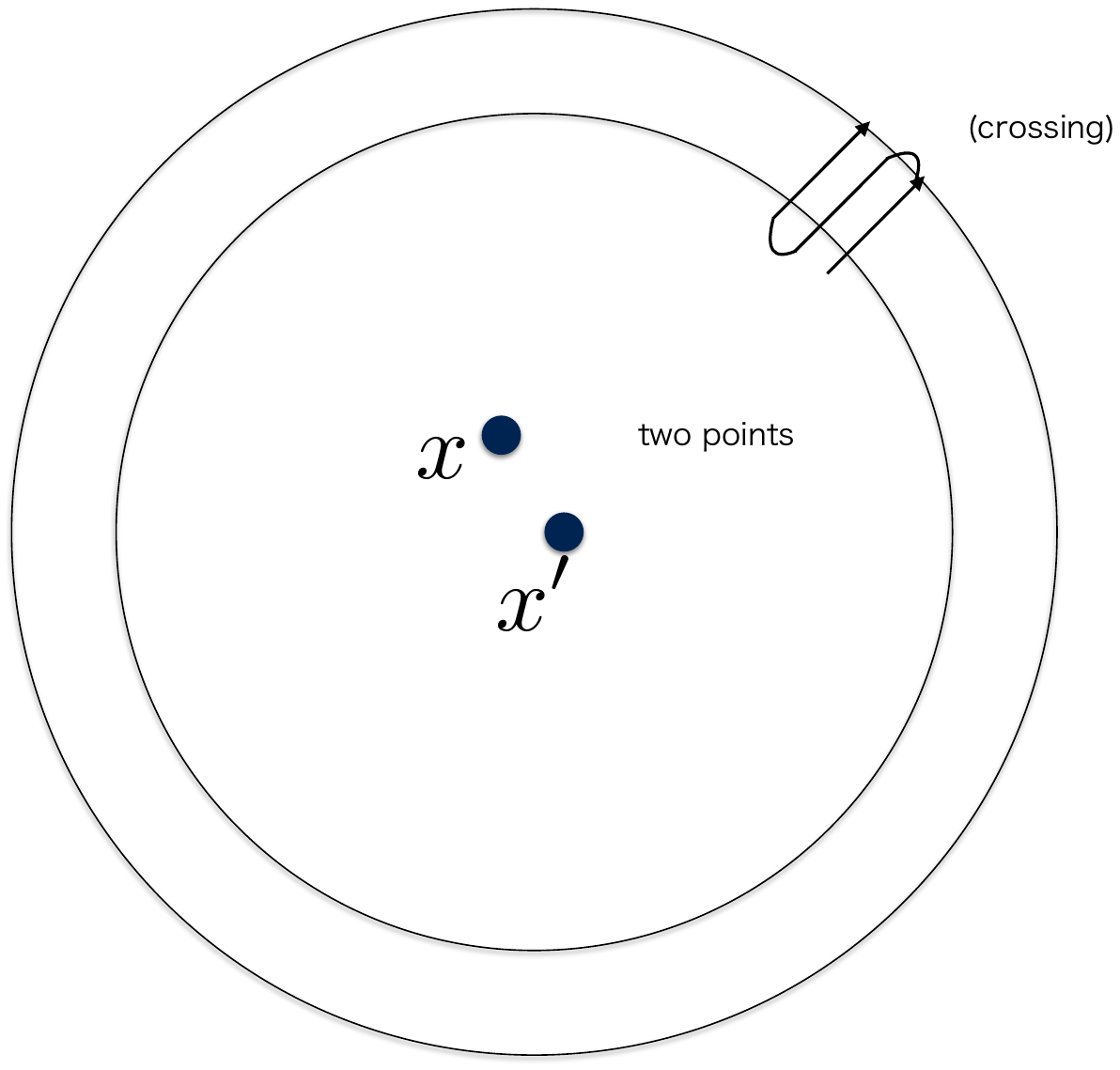}
\caption{}
\end{figure}
In \cite{Dembo2}, this method was used to study late points. 
In contrast to \cite{Dembo1, Dembo2}, Proposition \ref{uu0} provides an estimate for the probability that two points in a disk are  $\alpha$-favorite points under a particular condition on the crossing number around circles. 
%After estimating this probability, we follow the argument in \cite{Dembo2}. 
The result of Proposition \ref{uu0} precisely corresponds to Lemma $6.1$ in \cite{Dembo2}. 
%and after having Proposition \ref{uu0}, %we use the same method as the proof of the upper bound in \cite{Dembo2}. 
However, the proof of Proposition \ref{uu0} requires an original argument that is different from that of Lemma $6.1$ in \cite{Dembo2}. 
%We finally estimate the upper bound of the probability of the event that two points are $\alpha$-favorite points conditioned by crossing numbers around two points.   

The proof of the lower bound in Theorem \ref{h1} relies on the second moment method developed in \cite{Dembo1, Dembo3, Dembo2}. 
%Now, we explain one of the reason why the lower bound in Theorem $1.1$ in \cite{Dembo} has been unsolved for a long time and also the method of the proof in  \cite{Dembo, rosen}, which is essentially the same as  that of \cite{Dembo1,  Dembo3, Dembo2}. 
%If we use the Paley-Zygmund inequality for $|\Psi_n(\alpha)|$, we need to estimate the second moment of it.  
%However, it is much larger than the square of the first moment and, hence, it is impossible to use directly the moment method for $|\Psi_n(\alpha)|$. 
%Then, to estimate the lower bound, they consider the set of points $\Psi'_n(\alpha)$ called successful points which is a slightly smaller set than $\Psi_n(\alpha)$. 
%(For example, see ($3.7$) in \cite{rosen}.) 
%In fact, the second moment of $|\Psi'_n(\alpha)|$ is of almost the same order as the square of the first moment of $|\Psi'_n(\alpha)|$ and 
%the first moment of $|\Psi'_n(\alpha)|$ is of the same order as that of $|\Psi_n(\alpha)|$. 
%Then, they finally use the moment method for $|\Psi'_n(\alpha)|$, and, hence, 
%the definition of $\Psi'_n(\alpha)$ naturally yields the lower bound of $|\Psi_n(\alpha)|$. 
%In this paper, to use this method, we choose a proper set corresponding to  $|\Psi'_n(\alpha)|$. 
In the proof of the lower bound in Theorem \ref{h1}, similar to \cite{Dembo1, Dembo3, Dembo2}, we determine whether a site in $\mathbb{Z}^2$ is ``successful" or ``qualified'' in terms of the crossing numbers around that site.  
However, our definitions of ``successful" and ``qualified'' are slightly different from those of Dembo et al.  
From the definition of these events, 
we can compute the number of sites that are successful (or qualified) and use the second moment method. 
%The best way to understand the method on an intuitive level is to compare with the behavior of a Branching Brownian Motion, a Branching Random Walk or similar models. 
This method produces a variety of limit theorems for functionals of the local time in two dimensions or similar models, as in \cite{Dembo1, Dembo3, Dembo2}. 
At this point, the methods and tools of Belius-Kistler  \cite{kistler} have been surpassed. 
A new multi-scale refinement of the second moment method is introduced. 
Roughly speaking, this provides a cleaner analogy between these kinds of problems and branching Brownian motion, and identifies a more natural set of scales (radii of the balls) and a more natural scaling for the excursion counts, showing that their square root can be thought of as a direct equivalent of the path of a  particle in branching Brownian motion. 

The remainder of this paper is organized as follows. 
%We will give two main theorems, that is, Theorems \ref{h1} and \ref{h2} where we estimate (\ref{sss1}).  
In Section \ref{mar1}, we discuss joint distributions of the local time of a general Markov chain, 
before Section \ref{gre1} discusses some estimates concerning Green's function.  
In Section \ref{Th h2}, we present the proof of Theorem \ref{h2}. 
In the proof of Theorem \ref{h2}, we estimate the probability that two points whose distance is specified are $\alpha$-favorite points. 
To show the lower bound in Theorem \ref{h2}, 
we use a general Markov chain rule given in Section \ref{mar1}. 
In Section \ref{key*}, we discuss the key lemma for proving the upper bound in Theorem \ref{h1},  
and apply the results in Section \ref{mar1} 
to estimate the probability conditioned by crossing numbers. 
In Section \ref{upp Th h1}, with the aid of the lemma in Section \ref{key*}, 
we prove the upper bound in Theorem \ref{h1} using an argument similar to that in \cite{Dembo2}.  
In Section \ref{low Th h1}, we prove the lower bound in Theorem \ref{h1}.

%%%%%%%%%%%%%%%%%%%%%%%%%%%%%%%%%%%%%%%%%%%%%%%%%%%%%%%%%%%%%%%%%%%%%%%%%%%%%%%%%%%%%%%%%%%%%%%%%%%%%%%%%%%%%%%%%%%%%%%%%%%%%%%%
 \section{Preliminary results}
 \subsection{Markov chain argument}\label{mar1}
In this section, we discuss some preliminary results regarding finite Markov chains, 
which are used when considering the transitions of a random walk in $\mathbb{Z}^2$ from one given set to another. 
We can reduce such a problem to that for a Markov chain by considering a sequence of hitting times.

Consider a general discrete-time Markov chain $(X_m)_{m=0}^{\infty}$ on a state space 
$\{1,2,\ldots,m_1+m_2+m_3\}$. 
Let $M_1:=\{1,\ldots,m_1\}$, $M_2:=\{m_1+1,\ldots,m_1+m_2\}$ 
and $M_3:=\{m_1+m_2+1,\ldots,m_1+m_2+m_3\}$. 
For any $A\subset \{1,2,\ldots,m_1+m_2+m_3\}$, let $T_A:=\inf \{m\ge1: X_m\in A\}$. 
%Let $g_{i,l}$ be the  transition probability of $(X_m)_{m=0}^{\infty}$.  
Set 
\begin{align*}
g_{i,l}:=P^{i}(\min_{s\in \{1,2,3\}}T_{M_s}=T_{M_l})
\end{align*}
for $l \in \{1,2,3\}$ and $i\in \{1,2,\ldots,m_1+m_2+m_3\}$. 
In addition, for any $j$, $h\in \{1,2,3\}$, let 
\begin{align*}
d_{j,h}:=%1_{\{i\in \{1,\ldots,M_j\},l \in \{1,\ldots,M_h\} \}} 
\max_{i\in M_j } g_{i,h}.
\end{align*}
For $i \in \{1,2,\ldots,m_1+m_2+m_3\}$, let $P^i$ be the law of 
$(X_m)_{m=0}^{\infty}$ starting at $i$. 
Now, we extend the definition of the local time. 
We denote the number of visits to a particular set $A\subset \{1,2,\ldots,m_1+m_2+m_3\}$ up to time $n$ 
by $K(n, A)$, that is, $\sum_{i=1}^n 1_{\{S_i \in A \}}$. 
\begin{lem}\label{mma}
For any $n_1$, $n_2\in \mathbb{N}$ and $j\in M_1$, 
\begin{align}
\notag
&P^j\bigg(
K(n_1+n_2,M_1)=n_1,
K(n_1+n_2,M_2)=n_2,
X_{n_1+n_2}\in M_2 \bigg)\\
\label{we11}
&\le \sum_{0\le i\le n_1 \wedge (n_2-1)} 
d_{1,1}^{n_1-i}d_{1,2}^{i+1}
\begin{pmatrix}
n_1\\
i
\end{pmatrix}
d_{2,1}^{i}d_{2,2}^{n_2-i-1}
\begin{pmatrix}
n_2-1\\
i
\end{pmatrix}
\end{align}  
and for any $j\in M_2$,
\begin{align}\notag
&P^j\bigg( K(n_1+n_2,M_1)=n_1,
K(n_1+n_2,M_2)=n_2, 
X_{n_1+n_2}\in M_2 \bigg)\\
\label{we22}
&\le \sum_{1\le i\le n_1 \wedge n_2} 
d_{1,1}^{n_1-i}d_{1,2}^{i} 
\begin{pmatrix}
n_1\\
i-1
\end{pmatrix}
d_{2,1}^{i}d_{2,2}^{n_2-i}
\begin{pmatrix}
n_2\\
i
\end{pmatrix}
\end{align}
where $0!=1$. 
If $m_1=m_2=1$, then for any $n_1$, $n_2\in \mathbb{N}$, 
\begin{align}
\notag
&P^1\bigg(K(n_1+n_2,\{1\})=n_1,
K(n_1+n_2,\{2\})=n_2,
X_{n_1+n_2}=2\bigg)\\
\label{we1**}
&= \sum_{0\le i\le n_1 \wedge (n_2-1)} 
g_{1,1}^{n_1-i}g_{1,2}^{i+1}
\begin{pmatrix}
n_1\\
i
\end{pmatrix}
g_{2,1}^{i}g_{2,2}^{n_2-i-1}
\begin{pmatrix}
n_2-1\\
i
\end{pmatrix}
.
\end{align}  
\end{lem}
\begin{proof}
We show that (\ref{we11}) holds by induction on $n_1+n_2$. 
In the case of $n_1=n_2=1$, both (\ref{we11}) and (\ref{we22}) are immediate. 
Fix $n\ge 2$ and assume that (\ref{we11}) and (\ref{we22}) hold for any $n_1$, $n_2\in \mathbb{N}$ with $n_1+n_2=n$. 
We want to show  (\ref{we11})  for any $\tilde{n}_1$, $\tilde{n}_2\in \mathbb{N}$ with 
$\tilde{n}_1+\tilde{n}_2=n+1$. 
First, we discuss the case in which $\tilde{n}_2=1$, 
and then we discuss the case where $\tilde{n}_2\ge 2$. 
When $\tilde{n}_2=1$,  for any $j\in M_1$,
\begin{align*}
&P^j \bigg(K(\tilde{n}_1+\tilde{n}_2,M_1)=\tilde{n}_1,
K(\tilde{n}_1+\tilde{n}_2,M_2)=\tilde{n}_2,
X_{\tilde{n}_1+\tilde{n}_2}\in M_2 \bigg)\\
=&P^j \bigg(X_m\in M_1  \text{ for }1\le m\le \tilde{n}_1+\tilde{n}_2-1,
X_{\tilde{n}_1+\tilde{n}_2}\in M_2 \bigg)\\
\le&d_{1,1}^{\tilde{n}_1+\tilde{n}_2-1}d_{1,2},
\end{align*}
and so the claim holds. 
Next, we discuss the case $\tilde{n}_2 \ge 2$. 
%Note that, by changing the role of $M_1$ and $M_2$ and by the induction assumption, 
The Markov property states that, for any $j\in M_1$, $\tilde{n}_2\ge2$, and $\tilde{n}_1\in \mathbb{N}$,
\begin{align*}
&P^j\bigg(K(\tilde{n}_1+\tilde{n}_2,M_1)=\tilde{n}_1,
K(\tilde{n}_1+\tilde{n}_2,M_2)=\tilde{n}_2,
X_{\tilde{n}_1+\tilde{n}_2}\in M_2 \bigg)\\
\le&d_{1,1}\max_{j\in M_1}
P^j\bigg(K(\tilde{n}_1+\tilde{n}_2-1,M_1)=\tilde{n}_1-1,
K(\tilde{n}_1+\tilde{n}_2-1,M_2)=\tilde{n}_2,
X_{\tilde{n}_1+\tilde{n}_2-1}\in M_2 \bigg)\\
&+d_{1,2}\max_{j\in M_2}
P^j\bigg(K(\tilde{n}_1+\tilde{n}_2-1,M_1)=\tilde{n}_1,
K(\tilde{n}_1+\tilde{n}_2-1,M_2)=\tilde{n}_2-1,
X_{\tilde{n}_1+\tilde{n}_2-1}\in M_2 \bigg)\\
\le&d_{1,1} \sum_{0\le i\le (\tilde{n}_1-1)\wedge(\tilde{n}_2-1)} 
d_{1,1}^{\tilde{n}_1-1-i}d_{1,2}^{i+1} 
\begin{pmatrix}
\tilde{n}_1-1\\
i
\end{pmatrix}
d_{2,1}^{i}d_{2,2}^{\tilde{n}_2-1-i}
\begin{pmatrix}
\tilde{n}_2-1\\
i
\end{pmatrix}
\\
&+ d_{1,2}\sum_{1\le i\le \tilde{n}_1\wedge (\tilde{n}_2-1)} 
d_{1,1}^{\tilde{n}_1-i}d_{1,2}^i 
\begin{pmatrix}
\tilde{n}_1-1\\
i-1
\end{pmatrix}
d_{2,1}^{i}d_{2,2}^{\tilde{n}_2-1-i} 
\begin{pmatrix}
\tilde{n}_2-1\\
i
\end{pmatrix}
\\
=&\sum_{0\le i\le \tilde{n}_1 \wedge (\tilde{n}_2-1)} 
d_{1,1}^{\tilde{n}_1-i}d_{1,2}^{i+1}
\begin{pmatrix}
\tilde{n}_1\\
i
\end{pmatrix}
d_{2,1}^{i}d_{2,2}^{\tilde{n}_2-i-1} 
\begin{pmatrix}
\tilde{n}_2-1\\
i
\end{pmatrix},
\end{align*}
where the second inequality comes from the inductive assumptions (\ref{we11}) and (\ref{we22}). %by changing the role of $M_1$ and $M_2$.  
This completes the proof of (\ref{we11}). 
As we can easily obtain (\ref{we22}) and (\ref{we1**}) using the same argument as for (\ref{we11}), we omit the proof. 
\end{proof}
To formulate additional statements, 
we present the extensional definition of the local time. 
We denote the number of visits to a particular edge of the vertices $A_1$ and $A_2 \subset \{1,2,\ldots,m_1+m_2+m_3\}$ up to time $n$ 
by $\tilde{K}(n, A_1,A_2)$, that is, $\sum_{i=0}^{n-1} 1_{\{S_i \in A_1, S_{i+1}\in A_2 \}}$. 
Let $\tilde{M}_2:=\{m_1+m_2+m_3+1,\ldots,m_1+m_2+m_3+\tilde{m}_2 \}$ 
and consider a Markov chain $(X_m)_{m=0}^{\infty}$ on a state space 
$\{1,2,\ldots,m_1+m_2+m_3+\tilde{m}_2 \}$. 
Assume that $P^j(T_{\tilde{M}_2}<T_{M_1})=1$ for any $j\in M_2$ and 
$P^j(T_{\tilde{M}_2}<T_{M_2}<T_{M_3})=1$ for any $j\in M_1$. 
Let $\tilde{T}_{M_2}:=\inf \{m\ge T_{\tilde{M}_2}:  S_m\in M_2  \}$ and
\begin{align*}
\tilde{d}_{1,1}
:=& \max_{j\in M_1} P^j(T_{M_1} <T_{M_2}),\\
\tilde{d}_{1,2}
:=& \max_{j\in M_1} P^j(T_{M_2} <T_{M_1}),\\
\tilde{d}_{2,1}
:=& \max_{j\in M_2} P^j(T_{\tilde{M}_2}<  T_{M_3},  T_{M_1} <\tilde{T}_{M_2} ),\\
\tilde{d}_{2,2}
:=& \max_{j\in M_2} P^j(T_{\tilde{M}_2}<  T_{M_3}, \tilde{T}_{M_2} <T_{M_1}).
\end{align*}
\begin{lem}
For any $n_1$, $n_2\in \mathbb{N}$ and for any $j\in M_1$,
\begin{align}\notag
&P^j\bigg( K(T_{M_3},M_1)=n_1,
\tilde{K}(T_{M_3},M_2,\tilde{M}_2)=n_2 \bigg)\\
\label{we1}
&\le \sum_{0\le i\le n_1 \wedge n_2} 
\tilde{d}_{1,1}^{n_1-i}\tilde{d}_{1,2}^{i} 
\begin{pmatrix}
n_1\\
i
\end{pmatrix}
\tilde{d}_{2,1}^{i}\tilde{d}_{2,2}^{n_2-i}
\begin{pmatrix}
n_2\\
i
\end{pmatrix}
\end{align} 
and for any $j\in M_2$, 
\begin{align}
\notag
&P^j\bigg(
K(T_{M_3},M_1)=n_1,
\tilde{K}(T_{M_3},M_2,\tilde{M}_2)=n_2 \bigg)\\
\label{we2}
&\le \sum_{1\le i\le n_1 \wedge n_2} 
\tilde{d}_{1,1}^{n_1-i}\tilde{d}_{1,2}^{i}
\begin{pmatrix}
n_1\\
i-1
\end{pmatrix}
\tilde{d}_{2,1}^{i}\tilde{d}_{2,2}^{n_2-i}
\begin{pmatrix}
n_2\\
i
\end{pmatrix}
.
\end{align}  
\end{lem}
%In addition, the same argument yields the corresponding result for (\ref{we2}), and, hence, we obtain the desired result. 
\begin{proof}
The proof of the following lemma is almost the same as that of Lemma \ref{mma}, and is therefore omitted. 
We let $\tilde{d}_{i,l}$ play the role of  $d_{i,l}$ in Lemma \ref{mma}. 
To obtain (\ref{we1}), we  substitute $n_2+1$ for $n_2$ in (\ref{we11}), and 
to obtain (\ref{we2}), we retain $n_1$ and $n_2$ in (\ref{we22}). 
\end{proof}

\subsection{Hitting probabilities}\label{gre1}

In this section, we discuss an estimate concerning the hitting probability of a simple random walk in $\mathbb{Z}^2$. 

First, we compute the probability that the simple random walk does not hit two points until a particular random time.  
Let $P^x$ denote the law of a simple random walk starting at $x$. 
We simply write $P$ for $P^0$.  
For any $D\subset \mathbb{Z}^2$, let $T_D:=\inf \{m\ge1: S_m\in D\}$. 
For simplicity, we write $T_{x_1,\ldots,x_j}$ for $T_{\{x_1,\ldots,x_j\}}$. 
Given two distinct points $x_1$ and $x_2$ of $\mathbb{Z}^2$ and 
a nonempty subset of $A \subset \mathbb{Z}^2$ that does not contain $\{x_1,x_2\}$,  
define 
\begin{align*}
W_{i,l}:=\sum_{m=0}^\infty P^{x_i}(S_m= x_l ,m< T_A )
\end{align*}
for $i,l\in \{1,2\}$.
\begin{lem}
For any $i\neq l \in\{1,2\}$ and $A \subset \mathbb{Z}^2$,
\begin{align}
\label{hhu*}
P^{x_i}(T_A<T_{x_1,x_2})
=\frac{W_{l,l}-W_{i,l}}
{W_{1,1}W_{2,2}-W_{1,2}W_{2,1} }.
\end{align}
\end{lem}
\begin{proof}
Note that $P(T_A<\infty )=1$.  
Then, decomposing all of the events of the simple random walk starting at $x_i$ according to the time and the vertex in $\{x_1,x_2\}$ last visited before $T_A$ and taking the expectation, 
we obtain
\begin{align*}
1=
W_{i,1}P^{x_1}(T_A<T_{x_1,x_2})
+W_{i,2}P^{x_2}(T_A<T_{x_1,x_2}),
\end{align*}
for $i\in \{1,2\}$. 
As $W_{1,2}=P^{x_1}(T_{x_2}< T_A)W_{2,2}<W_{2,2}$
 and, similarly, $W_{2,1}<W_{1,1}$, 
we have
\begin{align*}
W_{1,1}W_{2,2}-W_{1,2}W_{2,1}\neq 0.
\end{align*}
Solving the above linear system gives the desired result. 
\end{proof}

Next, we recall some basic estimates of the hitting probability of a certain circle for a simple random walk in $\mathbb{Z}^2$. 
From \cite[Exercise $1.6.8$ or ($4.1$)]{Law}  and \cite[($4.3$)]{rosen}, 
we have that, in $0<r<|x|<R$,
\begin{align}\notag
P^x( T_0<\tau_R)&=\frac{\log (R/|x|)+O(|x|^{-1})}{\log R}(1+O((\log |x|)^{-1})),\\
\label{g2}
P^x( \tau_r<\tau_R)&=\frac{\log (R/|x|)+O(r^{-1})}{\log(R/{r})},
\end{align}
where $\tau_n=\inf\{m\ge 0: S_m \not\in D(0,n) \}$, as mentioned in Section 1. 
In addition, by \cite[Proposition $1.6.7$]{Law} or \cite[($2.2$)]{rosen}, we have that, 
for any $x\in D(0,n)$, 
\begin{align*}%\label{gg2*}
\sum_{m=0}^\infty P^x(S_m=0,m<\tau_n)=
\frac{2}{\pi}\log \bigg(\frac{n}{|x| \vee 1} \bigg)+O(1). 
\end{align*}
In particular, for $x,y \in D(0,n/3)$,
\begin{align}\label{gg2*}
G_n(x,y)
=\frac{2}{\pi}\log \bigg(\frac{n}{d(x,y)^+}\bigg)+O((d(x,y)^+)^{-1}+n^{-1}+1).
\end{align}
Indeed, for $x,y \in D(0,n/3)$,
\begin{align*}
G_n(x,y)&\le\sum_{m=0}^\infty P^{x-y}(S_m=0,m<\tau_{4n/3})\\
&=\frac{2}{\pi}\log \bigg(\frac{n}{d(x,y)^+}\bigg)+O((d(x,y)^+)^{-1}+n^{-1}+1),\\
G_n(x,y)&\ge \sum_{m=0}^\infty P^{x-y}(S_m=0,m<\tau_{2n/3})\\
&=\frac{2}{\pi}\log \bigg(\frac{n}{d(x,y)^+}\bigg)+O((d(x,y)^+)^{-1}+n^{-1}+1).
\end{align*}
Moreover, the Markov property implies that 
\begin{align}\label{g2*}
P(\tau_n< T_0)
=(\sum_{m=0}^\infty P(S_m=0,m<\tau_n))^{-1}
=\frac{\pi}{2\log n}(1+o(1)).
\end{align}

%%%%%%%%%%%%%%%%%%%%%%%%%%%%%%%%%%%%%%%%%%%%%%%%%%%%%%%%%%%%%%%%%%
\section{Proof of Theorem \ref{h2}}\label{Th h2}
In this section, we present a proof of Theorem \ref{h2} by dividing it into lower and upper bounds. 
\subsection{Proof of the lower bound in Theorem \ref{h2}}
 First, we discuss the lower bound of Theorem \ref{h2}. 
To show Theorem \ref{h2}, 
we introduce the following notation to establish a correspondence with the framework of Section \ref{mar1}.  
%In this section, we show the lower bound in Theorem \ref{h2}. 
Let $\tilde{\alpha}:= \lceil 4\alpha (\log n)^2 /\pi \rceil$ 
and $x$, $x'$ be two distinct points in $D(0,n)$. 
Set 
\begin{align*}
(U_0,U_1,U_2,U_3):=(0,x,x', \partial D(0,n))
\end{align*}
and 
\begin{align*}
b_{i,l}:=P^{U_i}(\min_{s\in \{1,2,3\}}T_{U_s}=T_{U_l}),
\end{align*}
for $i\in \{0,1,2\}$ and $l\in \{1,2,3\}$, 
where $\partial G:=\{y\in \mathbb{Z}^2\setminus G: d(x,y)=1 \text{ for some }x\in G \}$ for $G\subset \mathbb{Z}^2$. 
The proof of the lower bound comprises the following two lemmas. 
\begin{lem}\label{eeu0}
If we assume $b_{1,1}\le b_{2,2}$, then
\begin{align*}
P(x,x' \in \Psi_n(\alpha))
\ge \max_{i\in \{1,2\}}b_{0,i} b_{1,2}\frac{1}{\tilde{\alpha}^2}
P^x(T_{x,x'}< \tau_n )^{2\tilde{\alpha}-2}.
\end{align*}
\end{lem}

\begin{lem}\label{eeeu0}
Set $s:=\log d(x,x')/\log n$. 
For any $\epsilon>0$, there exists some $c>0$ such that,
 for all sufficiently large $n\in \mathbb{N}$ 
and any $x$, $x' \in D(0,n^{1-\epsilon})$, 
\begin{align}
\label{eeeu1}
 \max_{i\in \{1,2\}}b_{0,i} \ge& c,\\
 \label{eeeu2}
 b_{1,2}\ge& \frac{c}{\log n},\\
 \label{eeeu3}
P^x(T_{x,x'}< \tau_n )^{2\tilde{\alpha}-2}
=&\exp\bigg(-\frac{\tilde{\alpha} \pi}{(2-s)\log n} +o(1)\log n\bigg).
\end{align}
\end{lem}
%all the other three cases can be discussed in the same way as we see below since we will independently use two assumptions. 
\begin{proof}[Proof of Lemma \ref{eeu0}]
For $p\ge2$ and $x\in \mathbb{Z}^2$, let
\begin{align*}
T^1_x:=\inf \{m\ge 1:S_m=x\},\quad 
T^p_x:=\inf\{m>T_x^{p-1}:S_m=x \}.
\end{align*}
%As a first step, we estimate $P(x,x'\in \Psi_n(\alpha))$.  
Note that  the time-reversal of a simple random walk
yields $b_{1,2}=b_{2,1}$. 
Then, if we substitute $1$, $2$, $\{3\}$, $n_1$, and $n_2$ 
for $x$, $x'$, $\partial D(0,n)$, $\tilde{\alpha}-1$, and $\tilde{\alpha}$, 
 (\ref{we1**}) gives
\begin{align*}
P(x,x' \in \Psi_n(\alpha))
\ge &b_{0,1}
P^x(x,x' \in \Psi_n(\alpha), 
T_{x'}^{\tilde{\alpha}}< T_x^{\tilde{\alpha}-1} )\\
\ge &b_{0,1}
\times \sum_{0\le i\le \tilde{\alpha}-1}
 b_{1,1}^{\tilde{\alpha}-1-i}
b_{1,2}^{i+1}
\begin{pmatrix}
\tilde{\alpha}-1\\
i
\end{pmatrix}
b_{2,1}^i b_{2,2}^{\tilde{\alpha}-1-i}
\begin{pmatrix}
\tilde{\alpha}-1\\
i
\end{pmatrix}
\\
= &b_{0,1} b_{1,2} 
 \times \sum_{0\le i\le \tilde{\alpha}-1}
 b_{1,1}^{\tilde{\alpha}-1-i}
b_{1,2}^i
\begin{pmatrix}
\tilde{\alpha}-1\\
i
\end{pmatrix}
b_{2,1}^ib_{2,2}^{\tilde{\alpha}-1-i}
\begin{pmatrix}
\tilde{\alpha}-1\\
i
\end{pmatrix} \\
\ge&b_{0,1}b_{1,2}  
\bigg(\max_{0 \le i \le \tilde{\alpha}-1}b_{1,1}^{\tilde{\alpha}-1-i}
b_{1,2}^i
\begin{pmatrix}
\tilde{\alpha}-1\\
i
\end{pmatrix}
 \bigg)^2\\
\ge  &b_{0,1} b_{1,2}\frac{1}{\tilde{\alpha}^2}
(b_{1,1}+b_{1,2} )^{2\tilde{\alpha}-2}\\
= &b_{0,1} b_{1,2}\frac{1}{\tilde{\alpha}^2}
P^x(T_{x,x'}< \tau_n )^{2\tilde{\alpha}-2}
\end{align*}
(see Figure 2).
\begin{figure}[h]
\centering
\includegraphics[width=110mm]{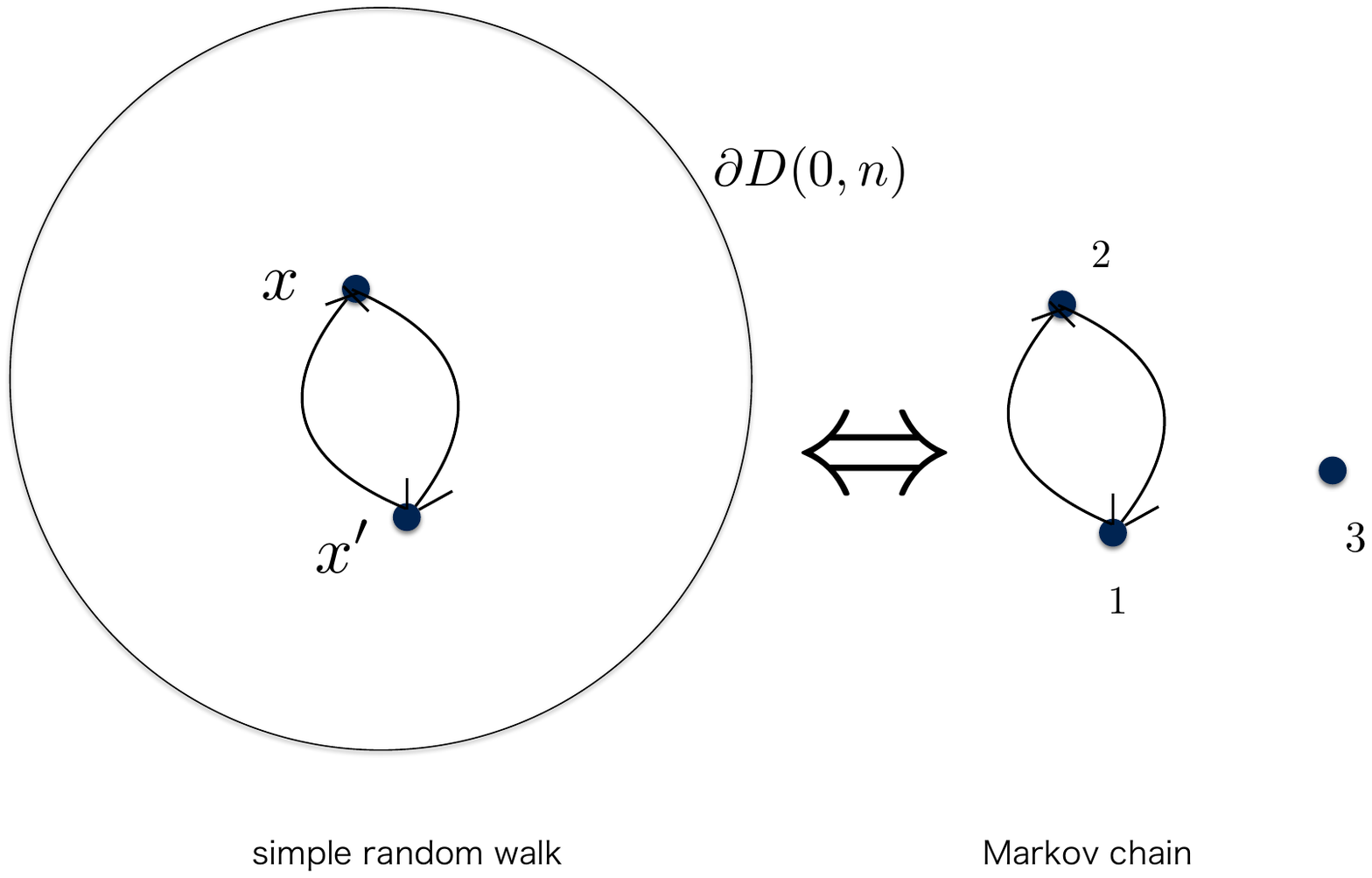}
\caption{}
\end{figure}
In addition, 
we also obtain 
\begin{align*}
P(x,x' \in \Psi_n(\alpha))
\ge &b_{0,2}
P^{x'}(x,x' \in \Psi_n(\alpha), 
T_{x}^{\tilde{\alpha}}< T_{x'}^{\tilde{\alpha}-1} )\\
\ge&b_{0,2}b_{2,1}  
\bigg(\max_{0 \le i \le \tilde{\alpha}-1}b_{1,1}^{\tilde{\alpha}-1-i}
b_{1,2}^i
\begin{pmatrix}
\tilde{\alpha}-1\\
i
\end{pmatrix}
 \bigg)^2\\
\ge &b_{0,2} b_{1,2}
\frac{1}{\tilde{\alpha}^2}
P^x(T_{x,x'}< \tau_n )^{2\tilde{\alpha}-2}.
\end{align*} 
Thus, we have the desired result. 
\end{proof}

\begin{proof}[Proof of Lemma \ref{eeeu0}]
%For the right-most number of (\ref{eeu0}), we estimate each of probabilities involved in it. 
First,  we show (\ref{eeeu3}). 
Assume $0<\epsilon<1$. 
Then, by (\ref{gg2*}), 
we have that  
\begin{align}\label{eeu}
\sum_{m=0}^{\infty}
P^{U_i}(S_m\in U_l,m<\tau_n)=
\begin{cases}
 \frac{2}{\pi}(1+o(1))\log n &\text{ if } i=l,\\
 \frac{2(1-s)}{\pi}(1+o(1))\log n &\text{ if } i\neq l
\end{cases}
\end{align}
uniformly in $x$, $x' \in D(0,n^{1-\epsilon})$, because $n/n^{1-\epsilon}\ge 3$  for all sufficiently large $n\in \mathbb{N}$. 
Hence,  if we set $T_A=\tau_n$ in (\ref{hhu*}), 
then 
\begin{align}\label{eeu1}
b_{1,3}
= \frac{\pi(1+o(1))}{2(2-s)\log n}.
\end{align}
Thus, uniformly in $x$, $x'$, 
\begin{align*}
P^x(T_{x,x'}< \tau_n )^{2\tilde{\alpha}}
= &\exp(-2\tilde{\alpha} b_{1,3}+o(1) \log n)\\
=&\exp\bigg(-\frac{\tilde{\alpha} \pi}{(2-s)\log n} +o(1)\log n\bigg).
\end{align*}
Hence, we obtain (\ref{eeeu3}). 
Next,  we confirm (\ref{eeeu1}) and (\ref{eeeu2}). 
Note that  (\ref{g2*}) yields 
\begin{align}\label{ff*}
P^x(\tau_n<T_x)=\frac{\pi}{2\log n}(1+o(1))
\end{align} 
that uniformly in $x$ and (\ref{g2}) implies that, 
for any $\epsilon>0$, there exists some $c>0$ such that, for all sufficiently large $n\in \mathbb{N}$,
\begin{align}\label{ff**}
&P^x(T_{x'}<\tau_n)
\ge P^x(T_{x'}<T_{\partial D(x', n/2)})
\ge \frac{\log (n/2n^{1-\epsilon})}{\log n/2}(1+O((\log n)^{-1}))
\ge c,\\
\label{ff***}
&P(T_x< \tau_n)
=P^x(T_0< \tau_{n})
\ge  \frac{\log (n/n^{1-\epsilon})}{\log n}(1+O((\log n)^{-1}))
\ge c,
\end{align}
where the equality comes from the time-reversal of a simple random walk. 
Note that 
\begin{align*} 
P^x(T_{x'}< \tau_n)
 =\sum_{i=0}^{\infty}b_{1,1}^ib_{1,2}
=\frac{1}{1-b_{1,1}}b_{1,2},
\end{align*}
and hence, by (\ref{ff*}) and (\ref{ff**}), for all sufficiently large $n\in \mathbb{N}$, 
\begin{align*}
b_{1,2}
\ge  P^x(\tau_n<T_x)P^x(T_{x'}< \tau_n)
\ge  \frac{c}{\log n}.
\end{align*}
Thus, we obtain (\ref{eeeu2}). 
In addition, by (\ref{ff***}), we have that there exists some $c>0$ such that, for any $n\in \mathbb{N}$, 
\begin{align*}
 \max_{i\in \{1,2\}}b_{0,i} 
\ge \frac{1}{2} P(T_{x,x'}< \tau_n)
\ge \frac{1}{2} P(T_x< \tau_n)
\ge  c.
\end{align*}
Therefore, we obtain (\ref{eeeu1}). 
\end{proof}

\begin{proof}[Proof of the lower bound in Theorem \ref{h2}]
%By the symmetry of $x$ and $x'$, we may assume that 
%\begin{align} \label{qwu}
%P(T_{x'}<T_x\wedge \tau_n)\le P(T_x<T_{x'}\wedge \tau_n). 
%\end{align}
%We now go back to (\ref{eeu0}). 
Using Lemmas \ref{eeu0} and \ref{eeeu0}, 
for any $\delta>0$ and $\epsilon>0$, there exists  $n_0\in\mathbb{N}$ such that, 
for $n\ge n_0$ and  $x$, $x' \in D(0,n^{1-\epsilon})$, 
\begin{align}\label{fi1}
P(x,x' \in \Psi_n(\alpha))
\ge \exp\bigg(-\frac{\tilde{\alpha} \pi}{(2-s)\log n} -\delta \log n\bigg).
\end{align}
%Now we turn to the desired lower bound. 
By (\ref{fi1}), for any $\delta>0$, $\epsilon>0$, 
 $\beta'\in (0,\beta]$, and  all sufficiently large $n\in \mathbb{N}$,
\begin{align*}
E[|\{ (x,x')\in \Psi_n(\alpha)^2:  d(x,x')\le n^{\beta} \}|] 
\ge &\sum_{\substack{x,x' \in D(0, n^{1-\epsilon}),\\ 0<d(x,x')\le n^{\beta'}} }
 P(x,x' \in \Psi_n(\alpha))\\
\ge &\sum_{\substack{x,x' \in D(0, n^{1-\epsilon}),\\ 0<d(x,x')\le n^{\beta'}} } 
\exp\bigg(-\frac{\tilde{\alpha} \pi}{(2-\beta')\log n} -\delta \log n\bigg)\\
\ge &n^{2+2\beta'-2\epsilon -4\alpha/(2-\beta' ) -2\delta}. 
%\ge n^{\hat{\rho}_2(\alpha, \beta)+o(1)}
\end{align*}
Note that an elementary calculation yields $\max_{\beta'\le \beta}2+2\beta' -4\alpha/(2-\beta' )=2+2\beta' -4\alpha/(2-\beta' )|_{\beta'=\beta \wedge (2-\sqrt{2\alpha})}=\hat{\rho}_2(\alpha,\beta)$. 
Then, as $\delta>0$ and $\epsilon>0$ are arbitrary, we obtain the desired result. 
\end{proof}

%%%%%%%%%%%%%%%%%%%%%%%%%%%%%%%%%%%%%%%%%%%%%%%%%%%%%%%%%%%%%%%%%%
\subsection{Proof of the upper bound in Theorem \ref{h2}}
We now prove the upper bound in Theorem \ref{h2}. 
The idea of the upper bound in Theorem \ref{h2} is similar to that of the lower bound, 
and we do not require a combinatorial argument in Section \ref{mar1}. 
First, we derive the probability that two particular points are $\alpha$-favorite points. 
\begin{lem}\label{qf}
For any $\delta>0$, there exists some $n_0\in \mathbb{N}$ such that, for any $n\ge n_0$ and  $x$, $x' \in D(0,n)$, 
\begin{align*}
P(x,x' \in \Psi_n(\alpha))
\le \exp\bigg(-\frac{\tilde{\alpha} \pi}{(2-s)\log n} +\delta \log n \bigg).
\end{align*}
\end{lem}
\begin{proof}
The strong Markov property implies that
\begin{align}
\notag
&P(x,x' \in \Psi_n(\alpha))\\
\notag
\le &
 P(|\{ l< \tau_n :S_l\in \{x',x\} \}| \ge  2 \tilde{\alpha})\\
\notag
 =& P(T_x<T_{x'} \wedge \tau_n )
P^x(|\{ l< \tau_n :S_l\in \{x',x\} \}| \ge  2 \tilde{\alpha}-1)\\
\notag
&\quad+
 P(T_{x'}<T_x \wedge \tau_n ) 
P^{x'}(|\{ l< \tau_n :S_l\in \{x',x\} \}| \ge  2 \tilde{\alpha}-1)\\
\notag
\le  &\max_{y\in \{x,x'\}}
P^y(|\{ l< \tau_n :S_l\in \{x',x\} \}| \ge  2 \tilde{\alpha}-1)
\\
\notag
\le & \max_{y' \in \{x,x'\}}
P^{y'}(T_{x,x'}<\tau_n)
\max_{y\in \{x,x'\}}
P^y(|\{ l< \tau_n :S_l\in \{x',x\} \}| \ge  2 \tilde{\alpha}-2)
\\
\notag
\le&\cdots
\\
\notag
\le & \max_{y' \in \{x,x'\}}P^{y'}(T_{x,x'}< \tau_n )^{2\tilde{\alpha}-1}
\\
\label{tre}
\le & \max_{y' \in \{x,x'\}}P^{y'}(T_{x,x'}< \tau_{2n} )^{2\tilde{\alpha}-1}.
\end{align}
Then, letting $T_A=\tau_{2n}$ in (\ref{hhu*}) 
and using the a similar argument to that for (\ref{eeu}) and (\ref{eeu1}), we obtain 
\begin{align*}
1-\max_{y' \in \{x,x'\}}P^{y'}(T_{x,x'}<\tau_{2n} )
= \frac{\pi(1+o(1))}{2(2-s)\log n} 
\end{align*}
uniformly in $x$, $x'\in D(0,n)$.
Hence, by (\ref{tre}), we have the desired result.
\end{proof}

\begin{proof}[Proof of the upper bound in Theorem \ref{h2}]
Compared with the lower bound, we need more estimates to bound the summation of 
$P(x,x' \in \Psi_n(\alpha))$ in $x' \in D(x,n^\beta)$. 
We divide the domain of $x'$  into  
$D(x,n^{\beta(1-\alpha)})$ and $D(x,n^{\beta(1-\alpha)})^c \cap D(x,n^\beta )$. 
We now consider the first case.  
As $\min\{1/\sqrt{\alpha},2/(2-\beta)\}\ge1$ with $0<\alpha, \beta<1$ 
and $F_{2, \beta}(\gamma)$ is decreasing in $\gamma \in [1, (1/\sqrt{\alpha}) \wedge (2/(2-\beta))]$, 
\begin{align}\label{jjt}
\hat{\rho}_2(\alpha,\beta)
\ge \rho_2(\alpha,\beta)
\ge 2+2\beta-2\alpha F_{2, \beta}(1)
=2(1-\alpha)+2\beta(1-\alpha).
\end{align}
In addition,  (\ref{g2*}) implies that, for any $\delta>0$, there exists some $n_0\in \mathbb{N}$ such that, for any $n\ge n_0$, 
\begin{align}\notag
E[| \Psi_n(\alpha) |]
=&\sum_{x\in D(0,n)}P(x\in\Psi_n(\alpha)  )\\
\notag
\le& \sum_{x\in D(0,n)}P^x(T_x<\tau_n  )^{\tilde{\alpha}-1}\\
\notag
\le& \sum_{x\in D(0,n)}P^x(T_x<T_{\partial D(x,2n)})^{\tilde{\alpha}-1}\\
\label{jjttt}
\le& 4n^2\times  \bigg(1-\frac{\pi}{2\log n}+o\bigg(\frac{1}{\log n}\bigg)\bigg)^{\tilde{\alpha}-1}
\le n^{2(1-\alpha)+\delta}. 
\end{align}
%Now we are ready to enter the main part of the proof. 
Using (\ref{jjttt}), we obtain
\begin{align}\notag
&E[|\{ (x,x')\in \Psi_n(\alpha)^2:d(x,x')\le n^{\beta(1-\alpha)} \}|] \\
\label{nnu*}
\le  &E[| \Psi_n(\alpha) |]\times Cn^{2\beta(1-\alpha)} 
\le n^{\hat{\rho}_2(\alpha, \beta)+\delta},
\end{align}
and hence, we obtain the first case. 
In addition, using Lemma \ref{qf},  we have that, for any $\delta>0$,  $0<h<1$, 
 $\beta'\in [\beta(1-\alpha)h^{-1},\beta]$, and all sufficiently large $n\in \mathbb{N}$,
\begin{align}\notag
&E[|\{ (x,x')\in \Psi_n(\alpha)^2: n^{\beta' h}<d(x,x')\le n^{\beta'} \}|] \\
\notag
= &\sum_{\substack{x,x' \in D(0, n),\\ n^{\beta' h}<d(x,x')\le n^{\beta'}} }
 P(x,x' \in \Psi_n(\alpha))\\
\notag
\le &\sum_{\substack{x,x' \in D(0, n), \\n^{\beta' h}<d(x,x')\le n^{\beta'} }} 
\exp\bigg(-\frac{\tilde{\alpha} \pi}{(2-\beta'h)\log n} +2\delta \log n\bigg )\\
\label{nnu**}
\le &n^{2+2\beta'-4\alpha/(2-\beta'h) +3\delta}.
\end{align}
Note that $\max_{(1-\alpha)\beta\le \beta'\le \beta}2+2\beta'-4\alpha/(2-\beta') =\hat{\rho}_2(\alpha, \beta)$. 
Hence, (\ref{nnu**}) yields 
\begin{align}\label{nnu***}
E[|\{ (x,x')\in \Psi_n(\alpha)^2: n^{\beta(1-\alpha)}<d(x,x')\le n^{\beta} \}|]
\le Cn^{\hat{\rho}_2(\alpha, \beta) +3\delta}.
\end{align}
Then, because $\delta>0$ and $0<h<1$ are arbitrary, 
(\ref{nnu*}) and (\ref{nnu***}) produce the desired result. 
\end{proof}

%%%%%%%%%%%%%%%%%%%%%%%%%%%%%%%%%%%%%%%%%%%%%%%%%%%%%%%%%%%%%%%%%%%%%%%%

\section{Key lemma for the proof of the upper bound in Theorem \ref{h1}}\label{key*}
This section presents large deviation estimates  that are the key to the proof of the upper bound in Theorem \ref{h1}. 
%These estimates are similar to Lemma $6.1$ in \cite{Dembo2} but methods of the proofs are different. 
We will estimate the probability of the event that two points are $\alpha$-favorite points conditioned by crossing numbers. 
Our method involves reducing a simple random walk path to a general Markov chain. This procedure is discussed in Section \ref{mar1}. 

We now present some definitions. 
Let $r_0:=0$ and $r_k:=(k!)^3$ for $k\in \mathbb{N}$.  
Set $K_n:=n^3r_n$  for $n\in \mathbb{N}$ and $f(n,k):=6\alpha(n-k)^2\log k$. 
Now, fix $0<\alpha,\beta<1$.  
For $z\in D(0,K_n)$, 
let 
$s_1:=T_{\partial D(z, r_{k})}$, 
\begin{align*}
t_i &:=\inf \{ m > s_i:  S_m \in \partial D(z,r_{k-1})\},\\
s_{i+1} &:=\inf \{ m > t_i :  S_m \in \partial D(z,r_k)\},
\end{align*}
and 
$N_{n,k}^z:= \sup \{i\in \mathbb{N}:   t_i <\tau_{K_n}  \}$ for $i\ge 1$.   
That is,  $N_{n,k}^z$   
denotes the number of excursions from $\partial D(z, r_{k})$ to $\partial D(z,r_{k-1})$ up to time $\tau_{K_n}$. 
To state the assertion, we define the domains $\mathbf{Go}^{\beta}$ and $\mathbf{Go}^{h,\beta}$ by 
\begin{align*}
\mathbf{Go}^{\beta}: &=\{(z,x): z \in D(0,K_n), x \in D(z, r_{\beta n-2})\cap D(0,K_n) \},\\
\mathbf{Go}^{h,\beta} : &=\{(z,x,x'): (z,x), (z,x') \in \mathbf{Go}^{\beta}, d(x,x') \ge  r_{\beta hn/2 -3} \}.
\end{align*}  
Let $\tilde{\Psi}_n(\alpha)
:=\{x:K(\tau_n,x) \in [4\alpha/\pi(\log n)^2, 4/\pi (\log n)^2]\}$. 
\begin{prop}\label{uu0} 
\quad 

$(1)$  For any $\delta>0$, there exists some $C>0$ such that, for any $\gamma \ge 0$ and $n\in \mathbb{N}$, 
\begin{align}\label{u1}
\max_{z\in {\mathbb{Z}}^2}P\bigg(\frac{N_{n,\beta n}^z}{f(n,\beta n)}\ge \gamma^2 \bigg)
\le CK_n^{-2\alpha F_{0,\beta}(\gamma)+\delta }.
\end{align}

$(2)$ For any $\delta>0$,  
there exist $C>0$,  $\delta_0>0$, and $0<h<2$ (with $h$ close to $2$) 
such that, for any $\gamma \in [0,1/\sqrt{\alpha}]$ and $n\in \mathbb{N}$, 
\begin{align}\label{u3}
\max_{(z,x,x')\in \mathbf{Go}^{h,\beta} }
P\bigg(x,x' \in \tilde{\Psi}_{K_n}(\alpha), 
\gamma^2 \le \frac{N_{n,\beta n}^z}{f(n,\beta n)} < (\gamma+\delta_0)^2 \bigg)
\le CK_n^{-2\alpha F_{2,\beta}(\gamma)+\delta }.
\end{align}
\end{prop}
From (\ref{u3}), 
we obtain the following bound, which we will use in a crucial way. 
\begin{cor}\label{u3+}
For any $\delta>0$, there exist $C>0$ and $\delta_0>0$ such that, for any $n\in \mathbb{N}$,
\begin{align*}
\max_{(z,x,x')\in \mathbf{Go}^{h,\beta} }
P\bigg(x,x' \in \tilde{\Psi}_{K_n}(\alpha), \frac{N_{n,\beta n}^z}{f(n,\beta n)}
< \bigg(\frac{1}{\sqrt{\alpha} }+\delta_0\bigg)^2\bigg)
\le CK_n^{-2\alpha F_{2,\beta}(\min\{1/\sqrt{\alpha},2/(2-\beta) \})+\delta }.
\end{align*}
\end{cor}
\begin{proof}
%Take $\delta_0>0$ satisfying
%2\alpha|F_{2,\beta}( \min\{1/\sqrt{\alpha},2/(2-\beta) \}+\delta_0)-F_{2,\beta}(\min\{1/\sqrt{\alpha},2/(2-\beta) \})|\le \delta.
%\end{align}
 If we set $C'=\lceil  1/(\sqrt{\alpha}\delta_0) \rceil$, 
(\ref{u3}) implies that 
\begin{align*}
&\max_{(z,x,x')\in \mathbf{Go}^{h,\beta} }
P\bigg(x,x' \in \tilde{\Psi}_{K_n}(\alpha),
\frac{N_{n,\beta n}^z}{f(n,\beta n)}
< \bigg(\frac{1}{\sqrt{\alpha} }+\delta_0\bigg)^2\bigg)\\
\le&
\sum_{k=1}^{C'-1}
\max_{(z,x,x')\in \mathbf{Go}^{h,\beta} }
P\bigg(x,x' \in \tilde{\Psi}_{K_n}(\alpha),
(k-1)^2 \delta_0^2 
\le \frac{N_{n,\beta n}^z}{f(n,\beta n)}
< k^2 \delta_0^2\bigg)\\
+&
\max_{(z,x,x')\in \mathbf{Go}^{h,\beta} }
P\bigg(x,x' \in \tilde{\Psi}_{K_n}(\alpha),
(C'-1)^2 \delta_0^2 
\le \frac{N_{n,\beta n}^z}{f(n,\beta n)}
< \bigg(\frac{1}{\sqrt{\alpha} }+\delta_0\bigg)^2\bigg)
\\
\le&C'\max_{\gamma \le 1/\sqrt{\alpha} } 
K_n^{-2\alpha F_{2,\beta}(\gamma)+\delta }\\
%=&C'K_n^{-2\alpha F_{2,\beta}(\min\{ 1/\sqrt{\alpha}+\delta_0,2/(2-\beta) \} )+\delta}\\
= &C'K_n^{-2\alpha F_{2,\beta}(\min\{1/\sqrt{\alpha},2/(2-\beta) \} )+\delta }.
\end{align*}
Here, the equality comes from the fact that $F_{2,\beta}$ is minimized at $2/(2-\beta)$. 
%and the last inequality is a consequence of the choice of $\delta_0$ in (\ref{gu}). 
Hence, we obtain the desired result.
\end{proof}

%%%%%%%%%%%%%%%%%%%%%%%%%%%%%%%%%%%%%%%%%%%%%%%%%%%%%%%%%%%%%%%%%%%%%%%%%%%%%%%%%%%%%%%%%%%%%%%%%%%%%%%%%%%%%%%%%%%%%%%%%%%%%%%%%%%%%%%%%%%%%%%%%%%%%%%%%%%%%%%%%%
Let us move to the proof of Proposition \ref{uu0}. 
The proof of the first assertion (\ref{u1}) is rather simple. 
\begin{proof}[Proof of Proposition \ref{uu0} (1)]
Substituting $R=K_n$ and $r=r_{\beta n-1}$ for (\ref{g2}), 
the strong Markov property implies that
\begin{align*}
\max_{z\in D(0,K_n)}P\bigg(\frac{N_{n,\beta n}^z}{f(n,\beta n)}\ge \gamma^2 \bigg)
\le &\max_{\substack{ z\in D(0,K_n),\\y\in \partial D(z,r_{\beta n}) }}
P^y(T_{\partial D(z,r_{\beta n-1})}<T_{\partial D(z,2K_n )}  )
^{ \lceil \gamma^2f(n,\beta n) \rceil-1}\\
\le &\max_{y\in \partial D(0,r_{\beta n}) }P^y(\tau_{r_{\beta n-1}}<\tau_{2K_n }  )
^{ \lceil \gamma^2f(n,\beta n) \rceil-1}\\
\le &\bigg(\frac{\log (2K_n/r_{\beta n})+O(r_{\beta n-1}^{-1})}{\log(2K_n/r_{\beta n-1})}\bigg)
^{\lceil \gamma^2f(n,\beta n) \rceil-1}\\
=&\bigg(1-\frac{1+o(1)}{n-\beta n}\bigg)
^{\lceil \gamma^2f(n,\beta n) \rceil-1}\\
=&\exp(-6\alpha \gamma^2 (n-\beta n) \log n+o(n\log n) )\\
=&CK_n^{-2\alpha F_{0,\beta}(\gamma)+o(1) },
\end{align*}
and hence, the desired result holds. 
\end{proof}

%%%%%%%%%%%%%%%%%%%%%%%%%%%%%%%%%%%%%%%%%%%%%%%%%%%%%%%%%%%%%%%%%%%%%%%%%%%%%%%%%%%%%%%%%%%%%%%%%%%%%%%%%%%%
To show the second assertion (\ref{u3}), we define and estimate the following probabilities.
For any $(z,x,x')\in \mathbf{Go}^{h,\beta}$, let 
\begin{align*}
(V_1,V_2):=(\partial D(z,r_{\beta n}), \{x,x'\})
\end{align*}
and, for $l= 1,2$,
\begin{align*}
a_{1,l}:&=
\max_{y\in \partial D(z,r_{\beta n-1})}
P^y(\min_{s\in \{1,2\} } T_{V_s}=T_{V_l}),\\
a_{2,l}:&=
\max_{y\in \{x,x'\}}
P^y(\min_{s\in \{1,2\} } T_{V_s}=T_{V_l} ). 
\end{align*} 
\begin{lem}\label{d1*}
For any $\epsilon>0$,  there exists some $0<h<2$ 
such that, for any $(z,x,x')\in \mathbf{Go}^{h,\beta}$ and all sufficiently large $n\in \mathbb{N}$, 
\begin{align}\label{d2}
a_{1,1}&\le 1-\frac{2-\epsilon}{\beta n}, \quad
a_{1,2}\le \frac{2+\epsilon}{\beta n},  \\
\label{d2*}
a_{2,1}&\le  \frac{\pi+\epsilon}{6\beta n\log n}\quad \text{ and }\quad
a_{2,2}\le 1-\frac{\pi-\epsilon}{6\beta n\log n}.
\end{align}
\end{lem}

\begin{proof}
To obtain the results, we use the relation between the hitting probability and Green's functions from  Section \ref{gre1}.  
We fix $\epsilon>0$. 
First, we show (\ref{d2}). 
From \cite[$(6.16)$]{Dembo2}, we obtain the following result for all sufficiently large $n\in \mathbb{N}$: 
\begin{align*}%\label{g1*}
a_{1,1}\le 1-\frac{2-\epsilon}{\beta n}.
\end{align*}
In addition, %since $r_{\beta n-1}-r_{\beta n-3}\ge r_{\beta n-2}$ and $r_{\beta n}+r_{\beta n-3}\le r_{\beta n+1}$,  
because $D(x,r_{\beta n-1}/2)\subset D(z,r_{\beta n-1})$ and 
$D(z,r_{\beta n})\subset D(x,2r_{\beta n})$ hold for $(z,x)\in \mathbf{Go}^{\beta}$, 
(\ref{g2}) implies that, for $y'\in \{x,x'\}$ and all sufficiently large $n\in \mathbb{N}$,
\begin{align*}
\max_{y\in \partial D(z,r_{\beta n-1})}
P^y(T_{y'}<T_{\partial D(z, r_{\beta n})}  )
\le& \max_{y\in  D(y',r_{\beta n-1}/2)^c}
P^y(T_{y'}<T_{\partial D(y', 2r_{\beta n})}  )\\
=& \max_{y\in \partial D(y',r_{\beta n-1}/2)}
P^y(T_{y'}<T_{\partial D(y', 2r_{\beta n})}  )\\
\le&\frac{\log (4r_{\beta n}/r_{\beta n-1})+O(r_{\beta n-1}^{-1})}{\log r_{\beta n}}(1+o(1))\\
\le &\frac{1+\epsilon}{\beta n}.
\end{align*}
As $a_{1,2}\le \max_{y\in \partial D(z,r_{\beta n-1})}
P^y(T_{x}<T_{\partial D(z, r_{\beta n})}  )
+\max_{y\in \partial D(z,r_{\beta n-1})}
P^y(T_{x'}<T_{\partial D(z, r_{\beta n})}  )$, we have (\ref{d2}). 
Next, we estimate (\ref{d2*}).  
From (\ref{gg2*}) and (\ref{g2*}), for any $(z,x,x')\in \mathbf{Go}^{h,\beta}$ and $y,y'\in \{x,x'\}$, 
\begin{align}\label{hhu}
\sum_{m=0}^\infty P^{y}(S_m=y' ,m<T_{\partial D(z, r_{\beta n})})
\begin{cases}
=\frac{2+o(1)}{\pi} 3\beta n\log n \quad&\text{ if }y=y',\\
\le \frac{2+o(1)}{\pi} \frac{3\beta(2-h)}{2}n\log n\quad&\text{ if }y\neq y'
\end{cases}
\end{align}
because $r_{\beta n}/r_{\beta n-2}\ge 3$  for all sufficiently large $n\in \mathbb{N}$.
Hence, if we substitute $T_{\partial D(z, r_{\beta n})}$ into $T_A$ in (\ref{hhu*}), 
then (\ref{hhu}) and the choice of $h$ close to $2$ yield (\ref{d2*}) for all sufficiently large $n\in \mathbb{N}$. 
\end{proof}

%%%%%%%%%%%%%%%%%%%%%%%%%%%%%%%%%%%%%%%%%%%%%%%%%%%%%%%%%%%%%%%%%%%%%%%%%%%%%%%%%%%%%%%%%%%%%%%%%%%%%%%%%%%%%%%%%%%%
\begin{proof}[Proof of Proposition \ref{uu0} (2)]
To obtain the desired result, we first use the Markov chain method described in Section \ref{mar1} 
and then present some detailed computations with the aid of Lemma \ref{d1*}.    
%For any $z\in D(0,K_n)$, 
Let 
$$\tilde{T}:=
\inf\{ m>T_{\partial D(z, r_{\beta n-1})}: S_m \in \partial D(z, r_{\beta n}) \}$$
and
\begin{align*}
a'_{1,1}:&=\max_{y\in \partial D(z,r_{\beta n})}
P^y(T_{\partial D(z, r_{\beta n-1})}<\tau_{K_n},
\tilde{T}<T_{x_1,x_2}),\\
a'_{1,2}:&=\max_{y\in \partial D(z,r_{\beta n})}
P^y(T_{\partial D(z, r_{\beta n-1})}<\tau_{K_n},
\tilde{T}>T_{x_1,x_2}).
\end{align*}
%In addition, let ${\cal R}_{\beta n, m}^z$ be the time until completion of the first $m$ excursions from $\partial D(z,r_{\beta n})$ to  $\partial D(z,r_{\beta n-1})$.  
Then, the strong Markov property implies that, 
for any $c_1$, $c_2$, $c_3 \in \mathbb{N}$, 
\begin{align}\notag
&\max_{(z,x,x')\in \mathbf{Go}^{h,\beta} }
P(K(\tau_{K_n},x)=c_2, K(\tau_{K_n},x')=c_3, N_{n,\beta n}^z=c_1 )\\
\notag
\le &\max_{\substack{(z,x,x')\in \mathbf{Go}^{h,\beta},\\  y\in \{x,x'\} }}
P^{y}(K(\tau_{K_n},x)+K(\tau_{K_n},x')=c_2+c_3, N_{n,\beta n}^z=c_1 )\\
\label{v1-}
+ &\max_{\substack{(z,x,x')\in \mathbf{Go}^{h,\beta}, \\y\in \partial D(z,r_{\beta n})}}
P^{y}(K(\tau_{K_n},x)+K(\tau_{K_n},x')=c_2+c_3, N_{n,\beta n}^z=c_1 ).
\end{align}
Now, we present estimates necessary to justify the truncation using (\ref{we1}) and (\ref{we2}). 
The truncation involves excursion counts between concentric circles around each point. 
%We use weight $a_{2,l}$ (or $a'_{1,l}$) instead of $d_{i,l}$ in (\ref{we1}) or (\ref{we2}). 
We consider the path corresponding to $(X_m)_{m=0}^\infty$, $M_1=\{x,x'\}$, $M_2=\partial D(z, r_{\beta n})$, 
$\tilde{M}_2=\partial D(z, r_{\beta n-1})$, and $M_3=D(0,K_n)^c$ in (\ref{we1}) or (\ref{we2}). 
%Note that we consider excursions from $\partial D(z, r_{\beta n-1})$ to $\partial D(z,r_{\beta n})$ 
%as the event that $\{\tilde{X}_m\}_{m=0}^\infty$ hits to $2$ (see an example of events in figure $3$). 
%\begin{figure}[htbp]
%\begin{center}
%\includegraphics[width=11cm,height=6cm,clip]{mar2.eps}
 %\caption{}
%\end{center}
%\end{figure}  
Therefore, if we set $n_1=c_2+c_3$, 
$n_2=c_1$, then (\ref{we1}) and (\ref{we2}) yield
\begin{align*}
&\max_{y\in \{x,x'\}}
P^{y}(K(\tau_{K_n},x)+K(\tau_{K_n},x')=c_2+c_3, N_{n,\beta n}^z=c_1 )\\
\le&\sum_{0\le i\le c_1  \wedge (c_2+c_3) } 
{a'}_{1,1}^{c_1-i}{a'}_{1,2}^i
\begin{pmatrix}
c_1\\
i
\end{pmatrix}
a_{2,1}^ia_{2,2}^{c_2+c_3-i}
\begin{pmatrix}
c_2+c_3\\
i
\end{pmatrix}
\end{align*}
and 
\begin{align*}
&\max_{y\in \partial D(z,r_{\beta n})}P^{y}(K(\tau_{K_n},x)+K(\tau_{K_n},x')=c_2+c_3, N_{n,\beta n}^z=c_1 )\\
\le&\sum_{1\le i\le  c_1  \wedge (c_2+c_3)  } 
{a'}_{1,1}^{c_1-i}{a'}_{1,2}^i
\begin{pmatrix}
c_1\\
i-1
\end{pmatrix}
a_{2,1}^{i}a_{2,2}^{c_2+c_3-i}
\begin{pmatrix}
c_2+c_3\\
i
\end{pmatrix}
.
\end{align*}
Let $A$ denote the probability on the left-hand side of (\ref{v1-}).  
Then, the above estimate implies that
\begin{align}\label{v1}
A\le 2(c_1+1)\max_{0\le i\le c_1}
{a'}_{1,1}^{c_1-i}{a'}_{1,2}^i
\begin{pmatrix}
c_1\\
i
\end{pmatrix}
a_{2,1}^{i}a_{2,2}^{c_2+c_3-i}
\begin{pmatrix}
c_2+c_3\\
i
\end{pmatrix}
\end{align}
because, for all sufficiently large $n\in \mathbb{N}$, $c_1\le c_2+c_3$ holds. 
Now, the strong Markov property implies that, for any $i=1$, $2$,
\begin{align*}
a'_{1,i}&\le \max_{y\in \partial D(z,r_{\beta n})}
P^y(T_{\partial D(z, r_{\beta n-1})}<\tau_{K_n})
a_{1,i}. 
\end{align*}
Note that, for any $0\le i\le c_1$,
\begin{align}\label{v2}
{a'}_{1,1}^{c_1-i} {a'}_{1,2}^{i} 
\le (\max_{y\in \partial D(z,r_{\beta n})}
P^y(T_{\partial D(z, r_{\beta n-1})}<\tau_{K_n} ))^{c_1}
a_{1,1}^{c_1-i} a_{1,2}^i .
\end{align}

Now, we combine this last estimate with our assertion. 
For $\gamma\in [0, 1/\sqrt{\alpha}]$ and $\delta>0$, 
we take $c_1$, $c_2$, and $c_3$ such that 
$\gamma^2\le c_1/f(n,\beta n) <\gamma^2+\delta_0$, 
$c_2,c_3 \in [4\alpha (\log K_n)^2/\pi, 4(\log K_n)^2/\pi  ]$. 
%For all sufficiently large $n\in \mathbb{N}$, $c_1\le c_2+c_3-1$ holds. 
Then, similar to (\ref{u1}), (\ref{g2}) implies that  
there exists some $C>0$ such that, for all sufficiently large $n\in{\mathbb{N}}$, 
\begin{align}\label{v3}
 (\max_{y\in \partial D(z,r_{\beta n})}
P^y(T_{\partial D(z, r_{\beta n-1})}<\tau_{K_n} ))^{c_1}
\le CK_n^{-2\alpha F_{0,\beta}(\gamma)+\delta/4}.
\end{align}
Therefore, applying (\ref{v2}) and (\ref{v3}) to (\ref{v1}), we obtain 
\begin{align}\label{v4}
A\le CK_n^{-2\alpha F_{0,\beta}(\gamma)+\delta/4}
\max_{0\le i\le c_1}
a_{1,1}^{c_1-i}a_{1,2}^i
\begin{pmatrix}
c_1\\
i
\end{pmatrix}
a_{2,1}^{i}a_{2,2}^{c_2+c_3-i}
\begin{pmatrix}
c_2+c_3\\
i
\end{pmatrix}
.
\end{align}
Let $A'$ denote the right-hand side of (\ref{v4}). 
We will show later that $A'$ is bounded by 
\begin{align}\label{promise}
CK_n^{-2\alpha F_{2,\beta}(\gamma)+\delta/2}. 
\end{align}
Under this assumption, we have 
\begin{align*}
&\max_{(z,x,x')\in \mathbf{Go}^{h,\beta} }
P\bigg(x,x'\in \tilde{\Psi}_{K_n}(\alpha), 
\gamma^2\le \frac{N_{n,\beta n}^z}{f(n,\beta n)}< (\gamma+\delta_0)^2 \bigg)\\
\le & \sum_{\substack{\gamma^2\le c_1/f(n,\beta n)< (\gamma+\delta_0)^2, 
\\c_2,c_3 \in [4\alpha/\pi (\log K_n)^2, 4/\pi (\log K_n)^2 ]}}
\max_{(z,x,x')\in \mathbf{Go}^{h,\beta} }
P(K(\tau_{K_n},x)=c_2, 
K(\tau_{K_n},x')=c_3, 
N_{n,\beta n}^z=c_1 )\\
\le &CK_n^{-2\alpha F_{2,\beta}(\gamma)+\delta},
\end{align*}
and hence, we obtain the desired result. 

Finally, we show that (\ref{promise}) is indeed true. 
%With the aid of Lemma \ref{d1*} we estimate the maximal term in (\ref{v4}). 
%Since it suffices to prove sufficiently large $n\in \mathbb{N}$, 
%we consider $c_1\le c_2\wedge c_3 $. 
Let us define $g(i)$ by 
\begin{align*}
g(i):=& 
\bigg(1-\frac{2-\epsilon}{\beta n}\bigg)^{{c_1-i}}
\bigg(\frac{2+\epsilon}{\beta n }\bigg)^{i} 
\begin{pmatrix}
c_1\\
i
\end{pmatrix}
\\
&\times
\bigg(\frac{\pi+\epsilon}{6\beta n \log n}\bigg)^i
\bigg(1-\frac{\pi-\epsilon}{6\beta n\log n}\bigg)^{c_2+c_3-i}
\begin{pmatrix}
c_2+c_3\\
i
\end{pmatrix}
.
\end{align*}
Then, from Lemma \ref{d1*} and (\ref{v4}), we have 
$A'\le CK_n^{-2\alpha F_{0,\beta}(\gamma)+\delta/4}\max_{0\le i\le c_1} 
g(i)$. 
Note that, for any $1\le i\le c_1$, 
\begin{align*}
\frac{g(i-1)}{g(i)}
=&\frac{(1-\frac{\pi-\epsilon}{6\beta n\log n})(1-\frac{2-\epsilon}{\beta n})i(i-1)}
{(\frac{\pi+\epsilon}{6\beta n \log n})(\frac{2+\epsilon}{\beta n })(c_2+c_3-i)(c_1+1-i) }.
\end{align*}
%Similarly, we have $g(c_1)/g(c_1-1)<1$. 
By taking the growth order of $c_1$, $c_2$, and $c_3$ into account, 
we attain the maximum of $g(i)$ for $0\le i\le c_1$ at
\begin{align*}
i_0:=\bigg\lceil (1+o(1)) \sqrt{(c_2+c_3)c_1\frac{2+\epsilon}{\beta n }
\times \frac{\pi+\epsilon}{6\beta n \log n}}\bigg\rceil .
\end{align*}
In addition, the Stirling formula implies that, for any $\delta>0$, there exists some $\delta_0>0$ such that, for all sufficiently large $n\in \mathbb{N}$ and small  $\epsilon > 0$, 
\begin{align*}
g(i_0)
\le & 
K_n^{\delta/11}\bigg(1-\frac{2-\epsilon}{\beta n}\bigg)^{{c_1-i_0}}
\bigg(\frac{2+\epsilon}{\beta n }\bigg)^{i_0} 
c_1^{i_0}\bigg(1-\frac{i_0}{c_1}\bigg)^{-c_1}i_0^{-i_0}\\
&\times
\bigg(\frac{\pi+\epsilon}{6\beta n \log n}\bigg)^{i_0}
\bigg(1-\frac{\pi-\epsilon}{6\beta n\log n}\bigg)^{c_2+c_3-i_0}
(c_2+c_3)^{i_0}\bigg(1-\frac{i_0}{c_2+c_3}\bigg)^{-c_2-c_3}i_0^{-i_0}\\
\le & 
K_n^{\delta/10}\bigg(1-\frac{2-\epsilon}{\beta n}\bigg)^{{c_1-i_0}}
\bigg(1-\frac{i_0}{c_1}\bigg)^{-c_1}
\bigg(1-\frac{\pi-\epsilon}{6\beta n\log n}\bigg)^{c_2+c_3-i_0}
\bigg(1-\frac{i_0}{c_2+c_3}\bigg)^{-c_2-c_3}\\
\le & 
K_n^{\delta/9}
e^{2i_0}
e^{-2c_1/(\beta n)}
e^{-\pi (c_2+c_3)/(6\beta n \log n)}\\
\le & 
K_n^{\delta/8}
K_n^{2\sqrt{\pi(c_2+c_3)}\gamma(1-\beta)  /(\beta\log K_n) }
K_n^{-4\alpha \gamma^2(1-\beta)^2/\beta }
K_n^{-\pi (c_1+c_2)/(\beta(\log K_n)^2)}.
\end{align*}
The final term on the right-hand side is bounded by
\begin{align*}
K_n^{-4\alpha  (\sqrt{(c'_2+c'_3)/2}-\gamma(1-\beta))^2/\beta+\delta/8}
\le K_n^{-4\alpha (1-\gamma(1-\beta))^2/\beta +\delta/4},
\end{align*}
where $c'_2=c_2\pi/(4\alpha(\log K_n)^2)$ and $c'_3=c_3\pi/(4\alpha(\log K_n)^2)$. 
%The last inequality comes from $\gamma\le 2/(2-\beta)\le 1/(1-\beta)$ for $\gamma\in [0,2/(2-\beta)]$. 
Therefore, we have that $A'$ is bounded by
\begin{align*}
&CK_n^{-2\alpha F_{0,\beta}(\gamma)+\delta/4} 
 K_n^{-4\alpha (1-\gamma(1-\beta))^2 /\beta+\delta/4}\\
=&CK_n^{-2\alpha F_{2,\beta}(\gamma)+\delta/2},
\end{align*}
and hence, we arrive at the desired result. 
\end{proof}

%%%%%%%%%%%%%%%%%%%%%%%%%%%%%%%%%%%%%%%%%%%%%%%%%%%%%%%%%%%%%%%%%%%%%%%%%%%%%%%%%%%%%%%%%%%%%%%%%%%%%%%%%%%%%%%%%%%%
\section{Proof of Theorem \ref{h1}}\label{Th h1}
In this section, we prove Theorem \ref{h1}. 
Using the notation introduced at the beginning of Section \ref{key*}, we first prove the upper bound of Theorem \ref{h1}, and then confirm the lower bound. 
\subsection{Proof of the upper bound in Theorem \ref{h1}}\label{upp Th h1}
In this section, our goal is to show that
\begin{align*}
\limsup_{n\to \infty}\frac{\log 
|\{ (x,x')\in \Psi_n(\alpha)^2:d(x,x')\le n^{\beta} \}|}{\log n}
\le \rho_2(\alpha, \beta)
\quad \text{a.s.}
\end{align*}
It suffices to prove the following proposition. 
\begin{prop}\label{propq1}
For any $0<\alpha, \beta,\delta<1$, 
there exist $C>0$ and $\epsilon>0$ such that, 
for any $n\in \mathbb{N}$,
\begin{align*}
P(|\tilde{\Theta}_{\alpha,0,\beta,n}|
\ge K_n^{\rho_2(\alpha,\beta)+5\delta})
\le CK_n^{-\epsilon},
\end{align*}
where
\begin{align*}
\tilde{\Theta}_{\alpha, \beta_2, \beta_1,n}
:=\{ (x,x')\in \tilde{\Psi}_{K_n}(\alpha)^2: r_{(\beta_2n-3) \vee 0}  \le d(x,x')\le r_{(\beta_1n-3)\vee 0} \}.
\end{align*}
\end{prop}
We first present the proof of   the upper bound in Theorem \ref{h1} with the aid of Proposition \ref{propq1}. 
\begin{proof}[Proof of the upper bound in Theorem \ref{h1}]
If Proposition \ref{propq1} holds, 
the Borel-Cantelli lemma produces 
\begin{align}\label{vv0*}
|\{ (x,x')\in {\tilde{\Psi}}_{K_{l+1}}(\alpha)^2 :
 d(x,x')\le  r_{\beta l-3}  \}|
 \le K_n^{\rho_2(\alpha,\beta)+5\delta}
 \quad\text{ a.s.}
\end{align} 
Note that  \cite[Theorem $1.1$]{Dembo} ensures that, for all sufficiently large $n\in \mathbb{N}$,
\begin{align}\label{vv}
| \Psi_{K_n}(\alpha)\setminus \tilde{\Psi}_{K_n}(\alpha) |\le K_n^{\delta}\quad\text{ a.s.}
\end{align}
Again, by \cite[Theorem $1.1$]{Dembo},  we have  
$|\Psi_{K_n}(\alpha)|\le K_n^{2(1-\alpha)+\delta}$ 
for all sufficiently large $n\in \mathbb{N}$ a.s. Hence, (\ref{jjt}) gives 
\begin{align}\label{vv*}
|\Psi_{K_n}(\alpha)|\le K_n^{\rho_2(\alpha, \beta)+\delta}\quad\text{ a.s.} 
\end{align}
Then,  by (\ref{vv}) and (\ref{vv*}), for all sufficiently large $n\in \mathbb{N}$, we obtain
\begin{align}\label{vv1*}
|\{ (x,x'):x\in \Psi_{K_n}(\alpha),
x'\in  \Psi_{K_n}(\alpha)\setminus \tilde{\Psi}_{K_n}(\alpha),
 d(x,x')\le r_{\beta n-3}  \}|\le  K_n^{\rho_2(\alpha, \beta)+2\delta}
\quad\text{ a.s.}
\end{align}\notag
Hence, by (\ref{vv0*}), (\ref{vv}), and (\ref{vv1*}), we have 
\begin{align*}
&|\{ (x,x')\in \Psi_{K_n}(\alpha)^2:
 d(x,x')\le r_{\beta n-3}  \}|\\
\le 
&|\{ (x,x')\in \tilde{\Psi}_{K_n}(\alpha)^2:
 d(x,x')\le r_{\beta n-3}  \}|\\
+&2|\{ (x,x'):x\in \Psi_{K_n}(\alpha),
x'\in  \Psi_{K_n}(\alpha)\setminus \tilde{\Psi}_{K_n}(\alpha),
 d(x,x')\le r_{\beta n-3}  \}|\\
 +&|\{ (x,x')\in   (\Psi_{K_n}(\alpha)\setminus \tilde{\Psi}_{K_n}(\alpha))^2: 
 d(x,x')\le r_{\beta n-3}  \}|
 \le  K_n^{\rho_2(\alpha, \beta)+6\delta}
\quad\text{ a.s.}
\end{align*}
Note that $ \log r_{\beta n-3}/\log K_{n+1} \to \beta$ as $n\to \infty$. 
Then, %with the aid of the symmetry of $x$ and $x'$ 
for any $0<\epsilon<\beta$ and all sufficiently large $n\in \mathbb{N}$, we have that
\begin{align*}
|\{ (x,x')\in \Psi_{n}(\alpha)^2 :
 d(x,x')\le n^{\beta-\epsilon}  \}|
 \le
& |\{ (x,x')\in \Psi_{K_{l+1}}(\alpha)^2 :
 d(x,x')\le  K_{l+1}^{\beta-\epsilon} \}|\\
  \le
& |\{ (x,x')\in \Psi_{K_{l+1}}(\alpha)^2 :
 d(x,x')\le  r_{\beta l-3}  \}|\\
 \le &K_{l+1}^{\rho_2(\alpha, \beta)+6\delta}
 \le  n^{\rho_2(\alpha, \beta)+7\delta}
\quad\text{ a.s.}
\end{align*}
if we select $l\in \mathbb{N}$ such that $K_l\le n <K_{l+1}$. 
Again, note that $\beta\mapsto \rho_2(\alpha, \beta)$ is continuous. 
Therefore, we obtain the desired result. 
\end{proof}

To show that Proposition \ref{propq1} holds, we divide the assertion into two corresponding estimates for 
$\tilde{\Theta}_{\alpha, 0, \beta(1-\alpha),n}$ and 
$\tilde{\Theta}_{\alpha, \beta(1-\alpha),\beta, n}$, 
and introduce the following two lemmas. 
For the former case, we present the following lemma. 
\begin{lem}\label{hhr0}
There exist $C>0$ and $\epsilon>0$ such that, for any $n\in \mathbb{N}$,
\begin{align*}
P(|\tilde{\Theta}_{\alpha, 0, \beta(1-\alpha),n}|
 \ge K_n^{\rho_2(\alpha,\beta)+4\delta})
\le CK_n^{-\epsilon}.
\end{align*}
\end{lem}
For the latter, we use the following lemma. 
\begin{lem}\label{hh}
There exist $h<2$, $C>0$, and $\epsilon>0$ such that, for any $\beta' \in [\beta (1-\alpha),\beta]$ and $n\in \mathbb{N}$,
\begin{align*}
q_{n,\beta'}:=
P(|\tilde{\Theta}_{\alpha, \beta' h/2, \beta',n}|
 \ge K_n^{\rho_2(\alpha,\beta')+3\delta})
\le CK_n^{-\epsilon}.
\end{align*}
\end{lem}
We now prove Proposition \ref{propq1} using Lemmas \ref{hhr0} and \ref{hh}. 
\begin{proof}[Proof of Proposition \ref{propq1}]
%We observe that the assertion for $\tilde{\Theta}_{\alpha, \beta(1-\alpha),\beta, n}$ indeed follows from this lemma. 
Set $\beta_j=\beta(h/2)^j$ and  $l=\min \{j: \beta_j\le \beta(1-\alpha) \}$. 
%By Lemma \ref{hh} we have $q_{n,\beta_j}\to 0$ as $n\to \infty$ for $j=0,..,l-1$. 
Combining this with the monotonicity of $\beta \mapsto \rho_2(\alpha,\beta)$ and 
the fact that $K_n^{\rho_2(\alpha,\beta)+4\delta} \ge l K_n^{\rho_2(\alpha,\beta)+3\delta}$ for all sufficiently large $n\in \mathbb{N}$, 
the union bound suggests that 
there exist $C>0$ and $\epsilon>0$ such that, for any $n\in \mathbb{N}$,
\begin{align*}
P(|\tilde{\Theta}_{\alpha, \beta(1-\alpha),\beta, n}|
 \ge K_n^{\rho_2(\alpha,\beta)+4\delta})
\le CK_n^{-\epsilon}.
\end{align*}
As $K_n^{\rho_2(\alpha,\beta)+5\delta} \ge 2K_n^{\rho_2(\alpha,\beta)+4\delta}$ for all sufficiently large $n\in \mathbb{N}$,  
with the aid of Lemma \ref{hhr0}, we have the union bound, and hence, Proposition \ref{propq1}. 
\end{proof}
Finally, we provide the proofs of Lemmas \ref{hhr0} and \ref{hh}. 
\begin{proof}[Proof of Lemma \ref{hhr0}]
Note that 
\begin{align*}
|\tilde{\Theta}_{\alpha, 0, \beta(1-\alpha),n}|
 \le 4r_{\beta(1-\alpha) n-3}^2|\tilde{\Psi}_{K_n}(\alpha)|
 \le K_n^{2\beta(1-\alpha)+\delta}|\Psi_{K_n}(\alpha)|.
\end{align*}
In addition, (\ref{jjt}) yields
\begin{align*}
P(|\tilde{\Theta}_{\alpha, 0, \beta(1-\alpha),n}|
 \ge K_n^{\rho_2(\alpha,\beta)+4\delta})
\le
P(|\tilde{\Theta}_{\alpha, 0, \beta(1-\alpha),n}|
 \ge K_n^{2(1-\alpha)+2\beta(1-\alpha) +4\delta}). 
\end{align*}
Hence, by  \cite[($2.18$)]{rosen},
the right-hand side in the last formula  is bounded by 
\begin{align*}
P(|\tilde{\Psi}_{K_n}(\alpha)|
 \ge K_n^{2(1-\alpha)+3\delta})
\le
P(|\Psi_{K_n}(\alpha)|
 \ge K_n^{2(1-\alpha)+3\delta})
\le CK_n^{-\epsilon},
\end{align*}
and hence,  we obtain the desired result. 
\end{proof}
\begin{proof}[Proof of Lemma \ref{hh}]
%Let $D_{n,\beta'}(x)$ denote the annulus $D(x,r_{\beta' n-3})\cap D(x,r_{\beta'h n/2-3})^c$, 
%that is, 
%\begin{align*}
%D_{n,\beta'}(x):=\{ x': r_{\beta' hn/2-3}  \le d(x,x')\le r_{\beta' n-3} \}.
%\end{align*}
%Let $\hat{Z}_{n,\beta'}$ denote the subgrid in $\mathbb{Z}^2_{K_n}$ 
%of spacing $4r_{\beta' n-4}$ and
Let $\hat{Z}_{n,\beta'}:=4r_{\beta' n-4}\mathbb{Z}^2 \cap D(0,K_n)$ and 
$z_{\beta'}(x)$ be the point closest to $x$ in $\hat{Z}_{n,\beta'}$. 
Fix $\delta>0$ and select $\delta_0>0$ as in Corollary \ref{u3+}. 
By a simple argument, 
\begin{align*}
q_{n,\beta'}
\le
I_1+I_2,
\end{align*}
where 
\begin{align*}
I_1&:=P\bigg( \max_{z\in\hat{Z}_{n,\beta'} }
\frac{ N_{n,\beta' n}^{z}}{n_{\beta' n}} 
\ge \bigg(\frac{1}{\sqrt{\alpha} }+\delta_0\bigg)^2\bigg),\\
I_2&:=P\bigg(|\tilde{\Theta}_{\alpha, \beta' h/2, \beta',n}|
 \ge K_n^{\rho_2(\alpha,\beta')+3\delta}
; \max_{z\in \hat{Z}_{n,\beta'}}\frac{N_{n,\beta' n}^{z}}{n_{\beta' n}}
< \bigg(\frac{1}{\sqrt{\alpha} }+\delta_0\bigg)^2\bigg).
\end{align*}
Then, 
\begin{align*}
I_2\le &P\bigg(\bigg|\bigg\{ (x,x')\in \tilde{\Psi}_{K_n}(\alpha)^2: r_{\beta' hn/2-3}  \le d(x,x')\le r_{\beta' n-3}, \frac{N_{n,\beta'n}^{z_{\beta'}(x)} }{n_{\beta' n} } 
< \bigg(\frac{1}{\sqrt{\alpha} }+\delta_0\bigg)^2 \bigg\}\bigg|
 \ge K_n^{\rho_2(\alpha,\beta')+3\delta} \bigg)\\
\le &K_n^{-\rho_2(\alpha,\beta')-3\delta} 
E\bigg[\bigg|\{ (x,x')\in \tilde{\Psi}_{K_n}(\alpha)^2: r_{\beta' hn/2-3}  \le d(x,x')\le r_{\beta' n-3}, 
\frac{N_{n,\beta'n}^{z_{\beta'}(x)}}{n_{\beta' n}}< \bigg(\frac{1}{\sqrt{\alpha} }+\delta_0\bigg)^2
\bigg\}\bigg|\bigg]\\
= &K_n^{-\rho_2(\alpha,\beta')-3\delta} 
\sum_{\substack{x\in D(0,K_n), \\x'\in D(x,r_{\beta' n-3})\cap D(x,r_{\beta'h n/2-3})^c}}
P\bigg(x,x'\in \tilde{\Psi}_{K_n}(\alpha); 
 \frac{N_{n,\beta'n}^{z_{\beta'}(x)}}{n_{\beta' n}}< \bigg(\frac{1}{\sqrt{\alpha} }+\delta_0\bigg)^2 \bigg)\\
\le &CK_n^{-2\alpha F_{2,\beta'}(\min\{1/\sqrt{\alpha},2/(2-\beta')\} )-2\delta} \\
&\times \max_{\substack{x\in D(0,K_n), \\x'\in D(x,r_{\beta' n-3})\cap D(x,r_{\beta'h n/2-3})^c}}
P\bigg(x,x'\in \tilde{\Psi}_{K_n}(\alpha); 
\frac{N_{n,\beta'n}^{z_{\beta'}(x)}}{n_{\beta' n}}< \bigg(\frac{1}{\sqrt{\alpha} }+\delta_0\bigg)^2 \bigg)\\
\le &CK_n^{-\delta}.
\end{align*}
The last inequality comes from Corollary \ref{u3+}. 
Set $\epsilon:=\alpha F_{0,\beta'}( 1/\sqrt{\alpha} +\delta_0 )-(1-\beta')$, 
and note that $\epsilon>0$. 
Then, there exists some $C>0$ such that
\begin{align*}
I_1
\le&
|\hat{Z}_{n,\beta'}|
CK_n^{-2\alpha F_{0,\beta'}( 1/\sqrt{\alpha} +\delta_0 )+\epsilon/2}\\
\le &
CK_n^{2-2\beta'-2\alpha F_{0,\beta'}( 1/\sqrt{\alpha} +\delta_0 )+\epsilon}
\le CK_n^{-\epsilon}.
\end{align*}
This completes the proof of Lemma \ref{hh}. 
\end{proof}

%%%%%%%%%%%%%%%%%%%%%%%%%%%%%%%%%%%%%%%%%%%%%%%%%%%%%%%%%%%%%%%%%%%%%%%%%%%%%%%%%%%%%%%%%%

\subsection{Proof of the lower bound in Theorem \ref{h1}}\label{low Th h1}
In this section, we show that 
\begin{align*}
\liminf_{n\to \infty}\frac{\log 
|\{ (x,x')\in \Psi_n(\alpha)^2:d(x,x')\le n^{\beta} \}|}{\log n}
\ge \rho_2(\alpha, \beta)
\quad \text{a.s.}
\end{align*}
To prove the lower bound in Theorem \ref{h1}, let $r_{n,k}:=r_n/r_k$ and 
\begin{align*}
\Theta_{\alpha, \beta, n,n'}
:=\bigg\{ (x,x')\in \Psi_{K_{n'}}(\alpha)^2: d(x,x')\le \frac{r_{n,(1-\beta )n}}{2} \bigg\}.
\end{align*}
Fix $\delta>0$, $0<\alpha,\beta<1$, and $\gamma\ge0$ with $\delta<\alpha$ and 
\begin{align}\label{ffh}
2-2\beta-2\alpha F_{0,\beta}(\gamma)>2\delta. 
\end{align}
Our principal goal in this section is to prove the following proposition. 
\begin{prop}\label{propq2}
There exist $\epsilon>0$ and $C>0$ such that, for any $n\in \mathbb{N}$, 
\begin{align*}
P(|\Theta_{\alpha-\delta,\beta,n,n+1}| \le K_n^{2+2\beta-2\alpha F_{2,\beta}(\gamma)-5\delta})\le Ce^{-\epsilon n}.
\end{align*}
\end{prop}
We first present the proof of   the lower bound in Theorem \ref{h1} using Proposition \ref{propq2}. 
\begin{proof}[Proof of the lower bound in Theorem \ref{h1}]
Under Proposition \ref{propq2}, 
the Borel-Cantelli lemma provides the lower bound in Theorem \ref{h1} for the sequence $(K_n)_{n\ge 1}$.  
In addition, note that  $\log r_{n,(1-\beta )n} / \log K_{n+1} \to \beta$ as $n\to \infty$. 
We have that, for any $0<\epsilon<\beta$ and all sufficiently large $n\in \mathbb{N}$,
\begin{align*}
|\{ (x,x')\in \Psi_{n}(\alpha)^2 :
 d(x,x')\le n^{\beta-2\delta}  \}|
 \ge
& |\{ (x,x')\in \Psi_{K_{l+1}}(\alpha)^2 :
 d(x,x')\le  K_{l+1}^{\beta-\delta} \}|\\
  \ge
& |\{ (x,x')\in \Psi_{K_{l+1}}(\alpha)^2 :
 d(x,x')\le  r_{l,(1-\beta+\delta )l}   \}|\\
 \ge &K_l^{\rho_2(\alpha+\delta, \beta-\delta)-5\delta}
 \ge  n^{\rho_2(\alpha+\delta, \beta-\delta)-6\delta}
\quad\text{ a.s.}
\end{align*}
if we select $l\in \mathbb{N}$ such that $K_{l+1}\le n <K_{l+2}$. 
As $(\alpha, \beta)\mapsto \rho_2(\alpha, \beta)$ is continuous, we have the desired result. 
\end{proof}
To show that Proposition \ref{propq2} holds, we first provide the following notation and intuitive description. 
Let $\mathcal{A}_n$ be a maximal set of points in $D(0,4r_n)\setminus D(0,3r_n)$, 
where any two points are separated by a distance of greater than $4r_{n,(1-\beta )n-4}$. 
The primary idea of the proof of Proposition \ref{propq2} is to consider 
$\Theta_{\alpha-\delta, \beta, n,n+1}\cap D(z,r_{n,(1-\beta)n}/2)^2$ 
for each $z\in \mathcal{A}_n$, instead of the entire  
set of $\Theta_{\alpha-\delta, \beta, n,n+1}$, and to count the number of elements in these sets. 
By the choice of the radius, 
if both $x$, $x'\in D(z,r_{n,(1-\beta)n}/2)$ are ($\alpha-\delta$)-favorite points, then 
$(x,x') \in \Theta_{\alpha-\delta, \beta, n,n+1}$ automatically.  

As the secondary idea, we restrict any random sets to smaller ones  to decrease the variance of the numbers of elements in those sets.  
The restriction controls the number of visits to $D(z,r_{n,(1-\beta)n}/2)$ according to the number of excursions passing through annuli around $D(z,r_{n,(1-\beta)n}/2)$. 

For $z\in \mathcal{A}_n$ %, $x \in D(z,r_{n, (1-\beta)n}/2)$ 
and $1\le l \le (1-\beta) n$, 
let $\hat{N}_{n,l}^z$ 
be the number of excursions from $\partial D(z, r_{n,l})$ to $\partial D(z,r_{n,l-1})$ up to time $T_{\partial D(z,r_n)}$. 
Let $\hat{\mathcal{R}}_{k,m}^z=\hat{\mathcal{R}}_{k}^z(m)$ be the time required for the first 
$m$ excursions from $\partial D(z, r_{n,k})$ to $\partial D(z, r_{n,k-1})$. 
Let 
\begin{align*}
\tilde{W}_{m}^z=\tilde{W}^z(m)
:=\bigg|\bigg\{y\in D\bigg(z,\frac{r_{n,(1-\beta)n}}{2}\bigg): K(m, y)
\ge \frac{4(\alpha-\delta)}{\pi}(\log K_n)^2 \bigg\}\bigg|
\end{align*}
and
\begin{align*}
 \hat{H}_{k,m}^z:=
\{\tilde{W}_{\hat{\mathcal{R}}_{k,m}^z}^z  \ge K_n^{2\beta- 2\alpha (1-\gamma(1-\beta))^2/\beta-2\delta} \}.
\end{align*}
In particular, we let
\begin{align}\label{ju}
 H_k^z:=
\{\tilde{W}^z(\hat{\mathcal{R}}^z_{k}(\hat{N}^z_{n,k}))  
\ge K_n^{2\beta- 2\alpha(1-\gamma(1-\beta))^2/\beta -2\delta} \}.
\end{align}
The event $H^z_k$ describes the situation in which sufficiently many ($\alpha-\delta$)-favorite points exist in a neighborhood of $z$ when all visits to the neighborhood have finished before leaving $D(z,r_n)$.  
Set $\tilde{f}(k):=6\gamma^2\alpha k^2\log k$. 
We say that a point $z \in \mathcal{A}_n$ is $(n,\beta)$-successful  
if $|\hat{N}_{n,k}^z - \tilde{f}(k)|\le k$ 
for all $3\le k\le (1-\beta) n$ and 
$H_{n-\beta n}^z$ occurs. 
That is, the first condition restricts the number of excursions to a ``typical" value. 
This definition comes from a fact stated in \cite{Dembo1},  
although this definition is slightly different. 

Using the notion of successful points, we can reduce our problem to that for the number of successful points. 
We use the following lemma and show that Proposition \ref{propq2} holds under the relevant assumption. 
\begin{lem}\label{r1+}
There exists some $C\in(0,1)$ such that,   for $\gamma\in \Gamma_{\alpha,\beta}$,
\begin{align*}
P(|\{z\in \mathcal{A}_n:  z\text{ is }(n,\beta)\text{-successful} \}|
\ge 
 K_n^{2(1-\beta)- 2\gamma^2\alpha(1-\beta)-\delta})\ge C.
\end{align*}
\end{lem}
\begin{proof}[Proof of Proposition \ref{propq2}]
By definition,
\begin{align}
\notag
|\Theta_{\alpha-\delta,\beta,n,n}|
\ge&
\sum_{z\in \mathcal{A}_n} \tilde{W}^z(\hat{\mathcal{R}}_{n-\beta n}^z(\hat{N}^z_{n,(1-\beta)n}))^2\\
\label{tg}
\ge& 
|\{z\in \mathcal{A}_n:  z\text{ is }(n,\beta)\text{-successful} \}|
 K_n^{4\beta- 4\alpha (1-\gamma(1-\beta))^2/\beta-4\delta}.
\end{align}
Then, Lemma \ref{r1+} yields
\begin{align*}
P(|\Theta_{\alpha-\delta,\beta,n,n}| \le K_n^{2+2\beta-2\alpha F_{2,\beta}(\gamma)-5\delta})\le 1-C.
\end{align*}
As we have that $K_{n+1}/K_n\ge n^3$ for $n\in \mathbb{N}$, 
the same argument used for \cite[$(3.4)$]{rosen} gives
\begin{align*}
P(|\Theta_{\alpha-\delta,\beta,n,n+1}| \le K_n^{2+2\beta-2\alpha F_{2,\beta}(\gamma)-5\delta})
\le P(|\Theta_{\alpha-\delta,\beta,n,n}| \le K_n^{2+2\beta-2\alpha F_{2,\beta}(\gamma)-5\delta})^{n^3/2},
\end{align*}
and hence, we obtain the desired result. 
\end{proof}
\begin{rem}
We now explain  why Theorems \ref{h1} and \ref{h2} result in different exponents. 
In fact,  for $\gamma \in \Gamma_{\alpha,\beta}$, as $n\to \infty$,
\begin{align*}
P(|\{z\in \mathcal{A}_n:  z\text{ is }(n,\beta)\text{-successful} \}|
\ge 1)\to0.
\end{align*}
Then, (\ref{tg}) does not hold for $\gamma \not\in \Gamma_{\alpha,\beta}$. 
Hence, we need to restrict the domain of $\gamma$.  
In contrast, as for Theorem \ref{h2}, (\ref{tg}) implies that, for any $\delta>0$, $\gamma\ge 0$, and all sufficiently large $n\in \mathbb{N}$,
\begin{align*}
&E[|\{ (x,x')\in \Psi_n(\alpha)^2:d(x,x')\le K_n^{\beta} \}|]\\
\ge& E |\Theta_{\alpha-\delta,\beta,n,n}|\\
\ge &
E|\{z\in \mathcal{A}_n:  z\text{ is }(n,\beta)\text{-successful} \}|
 \times K_n^{4\beta- 4\alpha (1-\gamma(1-\beta))^2/\beta-4\delta}.
\end{align*}
As we do not need to restrict the domain of $\gamma$ in the final term on the right-hand side of the last formula, 
we have the difference. 
\end{rem}
Thus, it suffices to show Lemma \ref{r1+} to conclude the proof of the lower bound in Theorem \ref{h1}. 
To show Lemma \ref{r1+}, we divide the major part of the problem into an estimate of the probability of the event $H^z_{(1-\beta)n}$ and the effect of restricting the number of excursions according to the definition of a successful point. 
Lemma \ref{r1} is concerned with the former. 
The latter is treated in Lemma \ref{r2} in combination with Lemma \ref{r1}.  
\begin{lem}\label{r1}
As $n\to \infty$,
\begin{align*}
 P(\hat{H}_{n-\beta n,  \tilde{f}(1-\beta) n)-(1-\beta) n}^z)=1-o(1)
\end{align*}
uniformly in $z\in \mathcal{A}_n$. 
\end{lem}
\begin{lem}\label{r2}
There exists some $\delta_n> 0$ with $\lim_{n \to \infty} \delta_n=0$ such that
\begin{align}\label{qe1}
\overline{q}_n 
:=\inf_{x\in \mathcal{A}_n } 
P(x \text{ is } (n,\beta)\text{-successful})
\ge r_{(1-\beta )n}^{-(2\gamma^2\alpha+\delta_n)}.
\end{align}
Further, for $x,y \in D(z,r_{n,(1-\beta) n}/2)$, set
\begin{align*}
l(x,y):=\min \{l: D(x,r_{n,l}+1) \cap D(y,r_{n,l}+1)=\emptyset \} \wedge n.
\end{align*}
Then, there exists some $c>0$ such that, for all $x \neq y \in  \mathcal{A}_n$,
\begin{align}\label{qe2}
P(x \text{ and }y \text{ are } (n,\beta)\text{-successful})
\le c\overline{q}_n^2 r_{l(x,y)}^{2\gamma^2\alpha+\delta_{l(x,y)}}.
\end{align}
%where $l'(x,y)=\min \{j\ge1 ; D(x,r_{n,j})\cap D(y,r_{n,j})=\emptyset \}\wedge \beta n$ and $l'(x,y)\le \beta n$ when $x \neq y \in  {\cal A}_n$.
\end{lem}
We now show that Lemma \ref{r1+} holds using Lemma \ref{r2}. 
Lemma \ref{r1} is used in the proof of  Lemma \ref{r2}. 
\begin{proof}[Proof of  Lemma \ref{r1+}]
By the same argument as in \cite[$(2.11)$]{Dembo1}, 
there exists some $C>0$ such that 
 \begin{align*}
E[|\{z\in \mathcal{A}_n:  z\text{ is }(n,\beta)\text{-successful} \}|^2]
\le
CE[|\{z\in \mathcal{A}_n:  z\text{ is }(n,\beta)\text{-successful} \}|]^2,
\end{align*}
where we have used Lemma \ref{r2} instead of \cite[Lemma $2.1$]{Dembo1}. 
%Note that $C>1$ must hold by the Schwarz inequality. 
By the Paley--Zygmund inequality, we have the desired result.
\end{proof}

Thus, to conclude the proof of the lower bound, it suffices to show that Lemma \ref{r2} holds with the aid of Lemma \ref{r1}.  
Hence, we now prove Lemmas \ref{r1} and \ref{r2}. 
To prove Lemma \ref{r1}, we first consider three propositions, 
%Fix $\gamma>0$ with (\ref{ffh}). 
which require the following notation.  
For $z\in \mathcal{A}_n$, $x \in D(z,r_{n, (1-\beta)n}/2)$, and $(1-\beta) n+2\le l \le n$, 
$\hat{N}_{n,l}^{z,x}$ denotes the number of excursions from $\partial D(x, r_{n,l})$ to $\partial D(x,r_{n,l-1})$ up to time 
$\hat{\mathcal{R}}^z_{(1-\beta) n, \tilde{f}((1-\beta) n)-(1-\beta) n}$. 
For $(1-\beta) n \le l\le n$ and $\gamma<1/(1-\beta) $, set
\begin{align*}
\hat{f}(n,l):=
6\alpha \bigg( \frac{l-(1-\beta) n}{\beta}+\bigg(n-\frac{l-(1-\beta) n}{\beta}\bigg)\gamma(1-\beta) \bigg)^2\log l.
\end{align*}
For $z\in \mathcal{A}_n$ and $\eta\in \mathbb{N}\setminus \{1\}$, we say that $x\in D(z,r_{n,(1-\beta) n}/2)$ is $(n,\beta)$-qualified if 
\begin{align*}
|\hat{N}_{n,l}^{z,x} -\hat{f}(n,l) |\le n 
\quad\text{ for }(1-\beta)n+\eta \le l\le n.
\end{align*}
Now, we consider $n\in \mathbb{N}$ with $(1-\beta)n\ge \eta/2$ (see Figure 3).
\begin{figure}[h]
\centering
\includegraphics[width=130mm]{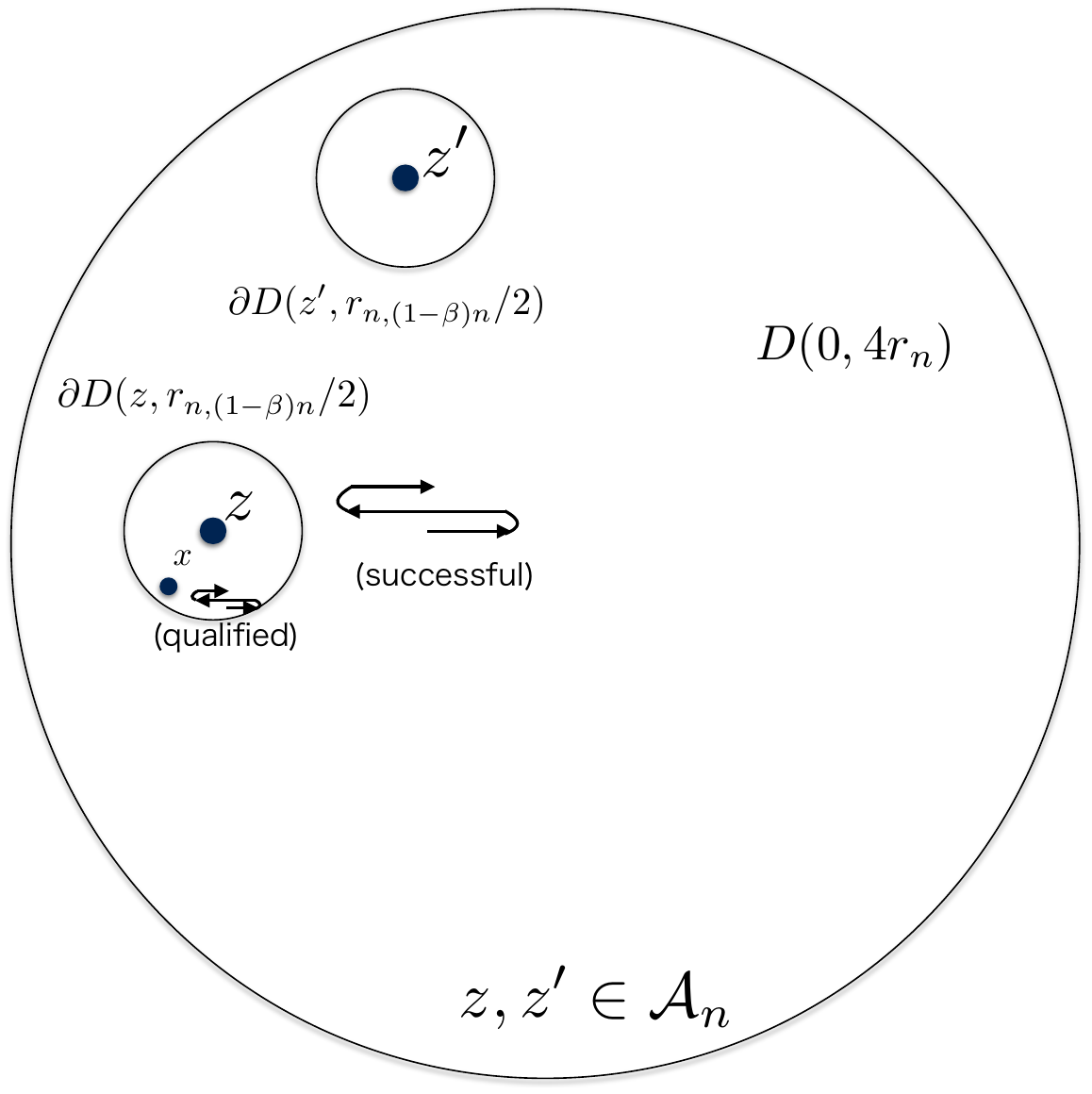}
\caption{}
\end{figure}   
\begin{rem}
In \cite{kistler}, a slightly different scaling for the excursion counts was considered. 
Roughly speaking, 
%they considered $r_k$ with $(\log r_k-\log r_{k-1} )/(\log r_{k+1}-\log r_{k-1})=1/2$ for $k\in \mathbb{N}$ and 
they expected $\sqrt{\hat{N}_{n,l}^{z,x}}$ to be an almost linear function in $l$. 
However, as we later use the same method as in \cite{abe, Dembo2, rosen}, we use almost the same definition of $r_k$. 
\end{rem}
\begin{prop}\label{r1*} 
As $n\to \infty$,
\begin{align*}
P\bigg(&\bigg\{x\in D\bigg(z,\frac{r_{n,(1-\beta) n}}{2}\bigg): x \text{ is }(n,\beta)
\text{-qualified}\bigg\}\\
\subset 
&\bigg\{x\in D\bigg(z,\frac{r_{n,(1-\beta) n}}{2}\bigg): 
\frac{K(\hat{\mathcal{R}}^z_{(1-\beta) n, \tilde{f}((1-\beta) n)-(1-\beta) n},x)}{(\log K_n)^2}\ge 
\frac{4\alpha}{\pi} -\frac{2}{\log \log  n} \bigg\}\bigg)\to 1.
\end{align*}
\end{prop}

\begin{prop}\label{r1**} 
For $x \in D(z,r_{n,(1-\beta) n}/2)$,
\begin{align*}
 P(x \text{ is }(n,\beta)\text{-qualified})=(1+o(1))q_n,
\end{align*}
where $q_n$ satisfies the following: 
there exist $c_1(\gamma,\alpha,\beta,\eta)$ and $c_2(\gamma,\alpha,\beta,\eta)>0$ 
such that, for all sufficiently large $n\in \mathbb{N}$, 
\begin{align*}
 e^{-c_1n \log \log n} K_n^{-2\alpha(1-\gamma(1-\beta))^2/\beta } 
\le q_n
\le  e^{-c_2n \log \log n} K_n^{-2\alpha(1-\gamma(1-\beta))^2/\beta } .
\end{align*}
\end{prop}

\begin{prop}\label{r1***}
$(i)$ There exist  
$c_3(\gamma,\alpha,\beta)$ and $c_4(\gamma,\alpha,\beta)>0$ 
such that, for all sufficiently large $\eta\in \mathbb{N}$, $x,y \in D(z,r_{n,(1-\beta)n}/2)$ 
with $(1-\beta )n+\eta \le l(x,y)\le n$, and sufficiently large $n\in \mathbb{N}$, 
\begin{align*}
 &P(x \text{ and }y \text{ are }(n,\beta)\text{-qualified})\\
\le &n^{c_3} e^{c_4(l-n+\beta n-\eta+1) \log \log n} q_n^2
\exp\bigg(\frac{6\alpha}{\beta}(1-\gamma(1-\beta))^2(l-n+\beta n-\eta+1)\log n\bigg).
\end{align*}
$(ii)$ For all $x,y \in D(z,r_{n,(1-\beta) n}/2)$ 
with $l(x,y)\le (1-\beta) n+\eta-1$, 
\begin{align*}
P(x \text{ and }y \text{ are }(n,\beta)\text{-qualified})=(1+o(1))q_n^2.
\end{align*}
\end{prop}
\begin{rem}\label{***}
Note that the choice of sufficiently large $\eta\in \mathbb{N}$ is used instead of the sufficiently large  $\overline{\gamma}>0$ in \cite{abe, Dembo2}. 
\end{rem}
Roughly speaking, 
Proposition \ref{r1*} ensures that the set of qualified points is essentially a restriction of 
$\tilde{W}^z_{\tilde{f}((1-\beta)n)-(1-\beta)n}$. 
Propositions \ref{r1**} and \ref{r1***} are used to apply the second moment method based on the Paley--Zygmund inequality. 
\begin{proof}[Proof of Proposition \ref{r1*}]
We refer to the proof of  \cite[Lemma $3.1$]{rosen} (see also \cite[Proposition $3.3$]{abe}).  
Note that,  for any $x\in D(z,r_{n,(1-\beta) n}/2)$,
\begin{align*}
&\bigg\{x \text{ is }(n,\beta)
\text{-qualified}\bigg\}
\cap 
\bigg\{\frac{K(\hat{\mathcal{R}}^z_{(1-\beta) n, \tilde{f}((1-\beta) n)-(1-\beta) n},x)}{(\log K_n)^2}\ge \frac{4\alpha}{\pi} -\frac{2}{\log \log  n} \bigg\}^c\\
\subset 
&\bigg\{ K(\hat{\mathcal{R}}^x_{n,\hat{f}(n,n)-n},x)  
\le \bigg(\frac{4\alpha}{\pi} -\frac{1}{\log\log n}\bigg)(\log K_n)^2\bigg\}.
\end{align*}
As there exists some $c>0$ such that $|D(0,K_n)|\le e^{cn \log n}$ 
for all sufficiently large $n\in \mathbb{N}$, 
it is sufficient to prove that  there exists some $c>0$ such that 
\begin{align*}
P\bigg(K(\hat{\mathcal{R}}^x_{n,\hat{f}(n,n)-n},x)  
\le \bigg(\frac{4\alpha}{\pi} -\frac{1}{\log\log n}\bigg)(\log K_n)^2\bigg)
\le \exp(-cn (\log n/\log_2 n)^2).
\end{align*}
This proof is the same as that of \cite[Lemma $3.1$]{rosen}. Thus, it is omitted. 
\end{proof}

\begin{proof}[Proof of Proposition \ref{r1**}]
%This proof is almost same as that of Proposition $3.4$ in \cite{abe}, ($4.13$) in \cite{rosen} or ($5.9$) in \cite{Dembo2}. 
Fix $x\in D(z,r_{n,(1-\beta) n}/2)$. 
Note that, from (\ref{g2}), we have that
\begin{align*}
P^{x_0}(T_{\partial D(x, r_{n, l+1})}< T_{\partial D(x, r_{n,l-1})})=\frac{1}{2}(1+O(n^{-8}))
\end{align*}
for $(1-\beta)n+\eta \le  l\le n-1$, 
$x_0 \in \partial D(x, r_{n,l})$, as well as  \cite[$(4.14)$]{rosen}. 
Then, for $x_0\in \partial D(x, r_{n, (1-\beta) n+\eta-1})$, 
\begin{align*}
 P^{x_0}(T_{\partial D(z, r_{n, (1-\beta)n-1})}
< T_{\partial D(x, r_{n, (1-\beta)n+\eta})} )=(1+O(n^{-8}))\frac{1}{\eta+1}
\end{align*}
(see Figure 4). 
\begin{figure}[h]
\centering
\includegraphics[width=130mm]{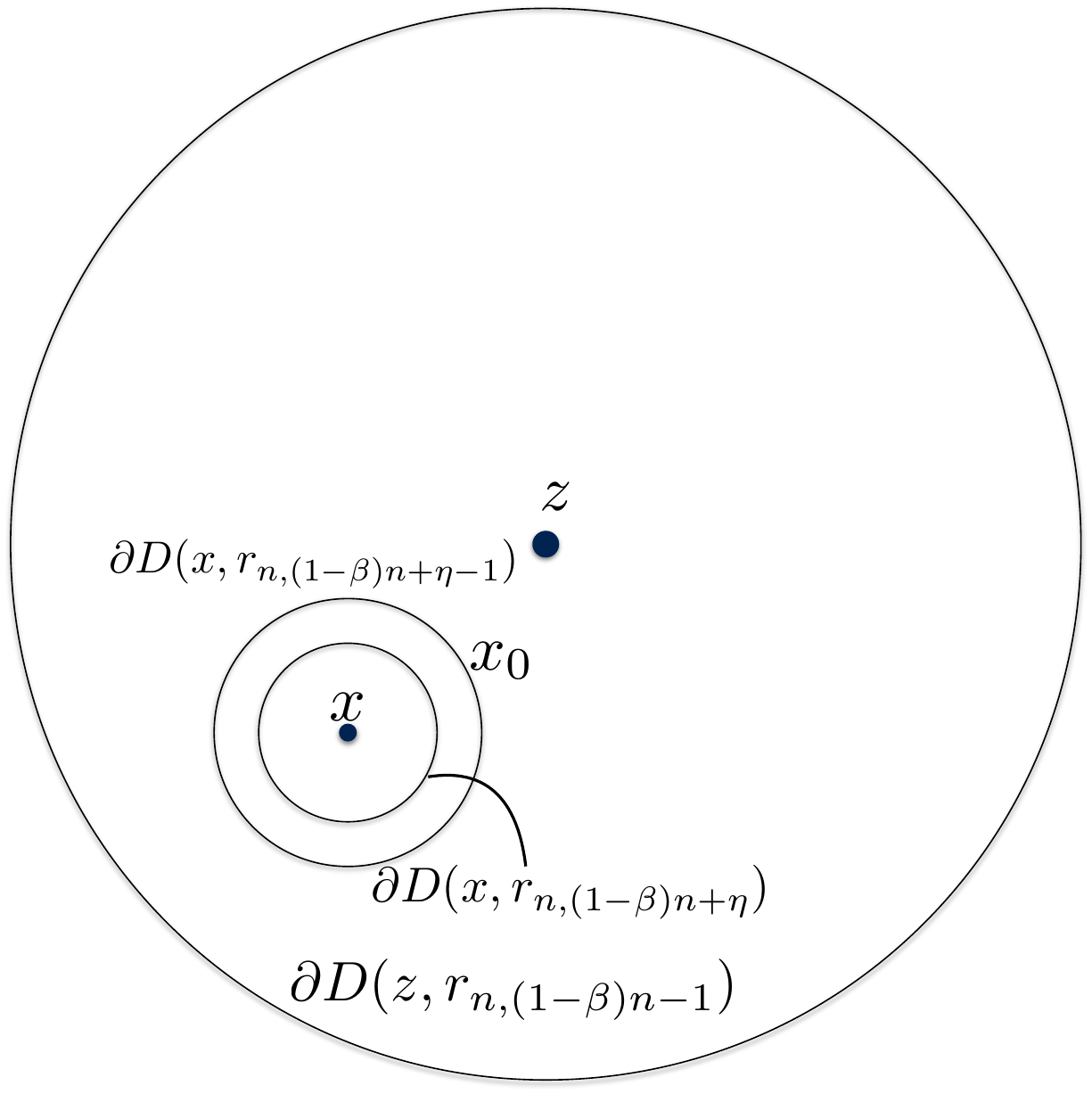}
\caption{}
\end{figure}
In addition, for $m_{(1-\beta)n+\eta}$ with 
$|m_{(1-\beta)n+\eta}-\hat{f}(n,(1-\beta)n+\eta)|\le n$, let
\begin{align*}
\hat{q}_n:=\sum_{1\le i \le m_{(1-\beta)n+\eta} \wedge (\tilde{f}((1-\beta)n)-(1-\beta) n-1) }
&\bigg(\frac{\eta}{\eta+1}\bigg)^{\tilde{f}((1-\beta) n)-(1-\beta)n-1-i}\bigg(\frac{1}{\eta+1}\bigg)^i\\
\begin{pmatrix}
\tilde{f}((1-\beta)n)-(1-\beta)n-1\\
i
\end{pmatrix}
&\bigg(\frac{1}{\eta+1}\bigg)^i\bigg(\frac{\eta}{\eta+1}\bigg)^{m_{(1-\beta)n+\eta}-i}
\begin{pmatrix}
m_{(1-\beta)n+\eta}-1\\
i
\end{pmatrix}
.
\end{align*}
Then, (\ref{we2}) and the same argument as for \cite[Proposition $3.4$]{abe}, \cite[ ($4.13$)]{rosen}, or \cite[($5.9$)]{Dembo2} imply that 
\begin{align}\label{f1}
 P(x \text{ is }(n,\beta)\text{-qualified})
=P(|\hat{N}_{n,l}^{z,x} -\hat{f}(n,l) |\le n \text{ for }(1-\beta)n+\eta \le l\le n)
=(1+o(1))q_n,
\end{align}
where
\begin{align}\label{f2}
q_n:=\sum \hat{q}_n
\times \prod_{l=(1-\beta)n+\eta+1}^n
 \left(
    \begin{array}{cc}
      m_l+m_{l-1}-1 \\
      m_l
    \end{array}
  \right)
\bigg(\frac{1}{2}\bigg)^{m_l+m_{l-1}-1}.
\end{align}
Here, the summations in (\ref{f1}) and (\ref{f2}) 
cover all $m_{(1-\beta)n+\eta},\ldots,m_{n}$ 
with $|m_l-\hat{f}(n,l)|\le n$ for $(1-\beta)n+\eta \le l\le n$. 
%Consider all sufficiently large $n\in \mathbb{N}$ with $\hat{f}(n,(1-\beta)n+\eta)+(1-\beta)n+\eta\le \tilde{f}(\beta n)-\beta n-1$. 
Hence, the Stirling formula implies that, for $m_{(1-\beta)n+\eta}$ 
with $|m_{(1-\beta)n+\eta}-\hat{f}(n,(1-\beta)n+\eta)|\le n$,
\begin{align}\notag
\hat{q}_n
\ge &\bigg(\frac{\eta}{\eta+1}\bigg)^{\tilde{f}(\beta n)-\beta n-1-i}\bigg(\frac{1}{\eta+1}\bigg)^i
\begin{pmatrix}
\tilde{f}(\beta n)-\beta n-1\\
i
\end{pmatrix}
\bigg(\frac{1}{\eta+1}\bigg)^i\bigg(\frac{\eta}{\eta+1}\bigg)^{m_{(1-\beta)n+\eta}-i}\\
\notag
&
\begin{pmatrix}
m_{(1-\beta)n+\eta}-1\\
i
\end{pmatrix}
\bigg|_{i=(2\eta\sqrt{ \tilde{f}(\beta n) m_{(1-\beta)n+\eta} }-
(m_{(1-\beta)n+\eta}+\tilde{f}(\beta n)-\beta n))/2(\eta^2-1) }\\
\label{qtt**}
=& n^{o(1)}.
\end{align}
By the same estimate as in the proof of  \cite[Proposition $3.4$]{abe}, 
we have that, for all $m_i $ 
with $|m_i-\hat{f}(n,i)|\le n$ and $(1-\beta)n+\eta+1\le l\le n$,
\begin{align}\notag
 &\left(
    \begin{array}{cc}
      m_l+m_{l-1}-1 \\
      m_l
    \end{array}
  \right)
\bigg(\frac{1}{2}\bigg)^{m_l+m_{l-1}-1}\\
\notag
%\ge &
%\frac{c_1}{\sqrt{m_l} }
%\frac{(m_l+m_{l-1}-1)^{m_l+m_{l-1} }} {(m_l)^{m_l} (m_{l-1})^{m_{l-1}} }
%\bigg(\frac{1}{2}\bigg)^{m_l+m_{l-1}-1}\\
%\notag
%=&c_1 (m_l)^{-1/2}\exp \bigg\{m_{l-1}f\bigg(\frac{m_l}{m_{l-1}}\bigg) \bigg\} \\
\ge & c (m_l)^{-1/2}
\exp\bigg\{-m_{l-1}\bigg(\frac{1}{4}\bigg(\frac{m_l}{m_{l-1}}-1\bigg)^2+c'\bigg|\frac{m_l}{m_{l-1}}-1\bigg|^3\bigg) \bigg\}\\
\notag
\ge & c (\log n)^{-cn}
\exp\bigg\{-
\frac{(m_l-m_{l-1})^2}{4m_{l-1}} \bigg\}\\
\notag
\ge & c (\log n)^{-cn}
\exp\bigg\{-
\frac{(6\alpha g_l^2 \log n
-6\alpha (g_l-(1/\beta-\gamma(1-\beta)/\beta ))^2 \log n)^2 }{4m_{l-1}} \bigg\}\\
\notag
\ge & c (\log n)^{-cn}
\exp\bigg\{-
\frac{( 12\alpha g_l(1/\beta-\gamma(1-\beta)/\beta )\log n )^2 }{4m_{l-1}} \bigg\}\\
\notag
\ge & c (\log n)^{-cn}
\exp\bigg\{-
\frac{6\alpha g_l(1/\beta-\gamma(1-\beta)/\beta )^2 \log n}
{g_{l-1}} \bigg\}\\
\label{qtt++} 
\ge & c (\log n)^{-cn}
\exp\bigg\{-\frac{6\alpha}{\beta^2}(1-\gamma(1-\beta))^2\log n \bigg\},
\end{align}
where $g_l:= l-(1-\beta) n/\beta+(n-l-(1-\beta) n/\beta)\gamma(1-\beta)$. 
%where $f(u):=(1+u)\log (1+u)-u\log u -(1+u)\log 2$, $u>0$ and 
%the second inequality comes from the Taylor expansion of $f$ around $1$. 
Therefore, we obtain 
\begin{align*}
q_n
%\ge &n^{\beta n}
%(c(\gamma)n^{-1}(\log n)^{-1/2})^{\beta n}
%\exp\bigg\{-\frac{6\gamma^2\alpha}{\beta}(1-\gamma(1-\beta))^2n\log n \bigg\}\\
\ge & c(\log n)^{-cn}
K_n^{-2\alpha(1-\gamma(1-\beta))^2/\beta},
\end{align*}
and hence, we have the desired lower bound.  
Note that (\ref{we1**}) makes it trivial that $\hat{q}_n\le 1$, and hence, 
\begin{align*}
\sum_{|m_{(1-\beta)n+\eta}-\hat{f}(n,(1-\beta)n+\eta)|\le n }\hat{q}_n\le n.
\end{align*}
If we use this instead of  (\ref{qtt**}) and a computation such as (\ref{qtt++}), 
we have the upper bound. 
\end{proof}
To prove Proposition \ref{r1***}, we set the following sigma field, as in \cite{abe, Dembo2}. 
First, we define a sequence of stopping times.  
Fix $z\in \mathcal{A}_n$ and $x\in D(z,r_{n, (1-\beta )n}/2)$. 
For $(1-\beta)n+\eta\le l\le n$, let 
$\sigma^{(0)}_x[l]:=0$ and 
\begin{align*}
\tau_x^{(0)}[l]&:=\inf \{m\ge 0: S_m \in \partial D(x,r_{n,l})\},\\
\sigma _x^{(i)}[l]&:=\inf \{m\ge \tau_x^{(i-1)}[l] : S_m \in \partial D(x,r_{n,l-1})\}\text{ for } i\ge 1,\\
\tau _x^{(i)}[l]&:=\inf \{m\ge \sigma_x^{(i)}[l] : S_m \in \partial D(x,r_{n,l})\}\text{ for } i\ge 1.
\end{align*}  
%Fix $x\in D(0,K_n/2)\cap D(0,r_n)^c$ and $(1-\beta)n+1\le l\le n$. 
%Set $e^{(0)}:= \{ S_m: 0\le m \le T_{\partial D(x,r_{n,l+1}) } \}$, and for $i\ge1$
%\begin{align*}
%e^{(i)}
%:=\{S_m%\sum_{j=0}^{i-1}\tau_x^{(j)}[r_{n,l+1},r_{n,l}] }
%: \sigma_x^{(i)}[r_{n,l+1},r_{n,l}] \le m \le \tau_x^{(i)}[r_{n,l+1},r_{n,l}] \}.
%\end{align*} 
In addition, let $\mathcal{G}_l^x:=\sigma (\bigcup_{i\ge 0} 
\{S_m :\sigma^{(i)}_x[l]\le  m \le\tau^{(i)}_x[l] \})$. 
Then, we obtain the following lemma. 
\begin{lem}\label{ghj}
There exists some $\epsilon_n$ with $\lim_{n\to \infty}\epsilon_n=0$ 
such that, for all $n-\beta n+\eta\le l\le n-1$  
with $|m_l-\hat{f}(n,l)|\le n$ and $x\in D(z,r_{n,(1-\beta) n}/2)$,
\begin{align*}
&P(\hat{N}_{n,i}^{z,x}=m_i \text{ for all } i=l,\ldots,n|\mathcal{G}_l^x)\\
=&(1+\epsilon_n)P(\hat{N}_{n,i}^{z,x}=m_i \text{ for all } i=l+1,\ldots,n|
\hat{N}_{n,l}^{z,x}=m_l)
1_{\{\hat{N}_{n,l}^{z,x}=m_l\}}.
\end{align*} 
\end{lem}
\begin{proof}
The proof is the same as that of \cite[Corollary $5.1$]{Dembo2}, 
because \cite[Lemma $2.4$]{Dembo2} 
holds even for a simple random walk in $\mathbb{Z}^2$ instead of one in a $2$-dimensional torus 
(see also  \cite[Lemma $3.6$]{abe}).  
\end{proof}
\begin{proof}[Proof of Proposition \ref{r1***}]
The proof  is almost the same as that of  \cite[Proposition $3.5$]{abe}, as the argument in \cite{abe} 
holds even for a discrete-time simple random walk in $\mathbb{Z}^2$ 
instead of the continuous-time one with Lemma \ref{ghj} 
(see also \cite[Lemma $3.2$]{rosen} or \cite[Lemma $4.2$]{Dembo2}).  
We now give an overview of the proof. 
As Proposition \ref{r1***}(i) yields Proposition \ref{r1***}(ii),  
we only prove the former. 
For $x,y \in D(z,r_{n,(1-\beta)n}/2)$ 
with $(1-\beta )n+\eta \le l(x,y)\le n$ and all sufficiently large $n\in \mathbb{N}$, 
Lemma \ref{ghj} gives
\begin{align*}
&P(x \text{ and }y \text{ are }(n,\beta)\text{-qualified})\\
\le&P(|\hat{N}_{n,i}^{z,x}-\hat{f}(n,i)|\le n \text{ for all } i=(1-\beta)n+\eta ,\ldots,l-3,l,\ldots,n\\
&\text{ and }|\hat{N}_{n,i}^{z,y}-\hat{f}(n,i)|\le n \text{ for all } i=l,\ldots,n)\\
\le&(1+o(1))\sum_{m_l:|m_l-\hat{f}(n,l)|\le n}
 P(|\hat{N}_{n,i}^{z,x}-\hat{f}(n,i)|\le n \text{ for all } i=(1-\beta)n+\eta ,\ldots,l-3,l,\ldots,n\\
&\text{ and }|\hat{N}_{n,i}^{z,y}-\hat{f}(n,i)|\le n \text{ for all } i=l+1 ,\ldots,n|\hat{N}_{n,l}^{z,y}=m_l ).
\end{align*}
Then, it suffices to show the following:    
\begin{align}
\notag
&\sum_{m_l:|m_l-\hat{f}(n,l)|\le n}
P(|\hat{N}_{n,i}^{z,y}-\hat{f}(n,i)|\le n\text{ for all } i=l+1 ,\ldots,n|\hat{N}_{n,l}^{z,y}=m_l )\\
\label{izu1}
\le &n^{c} e^{c(l-n+\beta n+1) \log \log n} q_n
\exp\bigg(\frac{6\alpha}{\beta}(1-\gamma(1-\beta))^2(l-n+\beta n-\eta+1)\log n\bigg),\\
\notag
&P(|\hat{N}_{n,i}^{z,x}-\hat{f}(n,i)|\le n \text{ for all } i=(1-\beta)n+\eta ,\ldots,l-3,l,\ldots,n )\\
\label{izu2}
\le &n^{c} e^{c(l-n+\beta n-\eta+1) \log \log n} q_n.
\end{align}
First, we prove (\ref{izu1}). 
By Proposition \ref{r1**} and  Lemma \ref{ghj}, 
\begin{align*}
&(1+o(1))q_n\\
=&P(y \text{ is }(n,\beta)\text{-qualified})\\
\ge &(1+o(1))
\sum_{m_l:|m_l-\hat{f}(n,l)|\le n}
P(|\hat{N}_{n,i}^{z,y}-\hat{f}(n,i)|\le n \text{ for all } i=(1-\beta)n+\eta ,\ldots,l-1; \hat{N}_{n,l}^{z,y}=m_l) \\
\times&P(|\hat{N}_{n,i}^{z,y}-\hat{f}(n,i)|\le n \text{ for all } i=l+1 ,\ldots,n|\hat{N}_{n,l}^{z,y}=m_l ).
\end{align*}
By the same argument as in the proof of Proposition \ref{r1**}, 
\begin{align*}
&P(|\hat{N}_{n,i}^{z,y}-\hat{f}(n,i)|\le n \text{ for all } i=(1-\beta)n+\eta ,\ldots,l-1; \hat{N}_{n,l}^{z,y}=m_l) \\
\ge &n^{-c} e^{c(l-n+\beta n+1) \log \log n} q_n
\exp\bigg(\frac{6\alpha}{\beta}(1-\gamma(1-\beta))^2(l-n+\beta n-\eta+1)\log n\bigg),
\end{align*}
and hence, we obtain (\ref{izu1}). 
Next, we show  (\ref{izu2}). 
By Lemma \ref{ghj},  (\ref{izu1}), and the same argument as the proof of Proposition \ref{r1**}, 
we have that, for all sufficiently large $\eta>0$, 
\begin{align*}
&P(|\hat{N}_{n,i}^{z,x}-\hat{f}(n,i)|\le n  \text{ for all } i=(1-\beta)n+\eta ,\ldots,l-3,l,\ldots,n )\\
\le &(1+o(1)) 
P(|\hat{N}_{n,i}^{z,x}-\hat{f}(n,i)|\le n  \text{ for all } i=(1-\beta)n+\eta ,\ldots,l-3 )\\
\times &\sum_{m_l:|m_l-\hat{f}(n,l)|\le n}
P(|\hat{N}_{n,i}^{z,x}-\hat{f}(n,i)|\le n  \text{ for all } i=l+1,\ldots,n | \hat{N}_{n,l}^{z,x}=m_l)\\
\le  &n^{c} e^{c(l-n+\beta n-\eta+1) \log \log n} 
\exp\bigg(-\frac{6\alpha}{\beta}(1-\gamma(1-\beta))^2(l-n+\beta n-\eta+1)\log n\bigg)\\
\times &n^{c} e^{c(l-n+\beta n-\eta+1) \log \log n} q_n
\exp\bigg(\frac{6\alpha}{\beta}(1-\gamma(1-\beta))^2(l-n+\beta n)\log n\bigg)\\
\le &n^{c} e^{c(l-n+\beta n-\eta+1) \log \log n} q_n.
\end{align*}
Therefore, we obtain the desired result. 
\end{proof}

\begin{proof}[Proof of Lemma \ref{r1}]
As mentioned in Remark \ref{***}, 
we consider the same argument as in the proof of the lower bound in  \cite[Theorem $1.2$(i)]{abe} (or \cite{Dembo2}). 
If we select $\eta$ such that $c_3<6\eta-7$ in Proposition \ref{r1***}, Propositions \ref{r1**} and \ref{r1***} state that,  for all sufficiently large $n\in \mathbb{N}$, 
 \begin{align*}
E\bigg[\bigg|\bigg\{x\in D\bigg(z, \frac{r_{n,(1-\beta)n}}{2}\bigg) :  x\text{ is }(n,\beta)\text{-qualified} \bigg\}\bigg|\bigg]
\ge (1+o(1))\bigg|D\bigg(z, \frac{r_{n,(1-\beta)n}}{2}\bigg)\bigg|q_n,
\end{align*}
 \begin{align*}
&E\bigg[\bigg|\bigg\{(x,y)\in D\bigg(z, \frac{r_{n,(1-\beta) n}}{2}\bigg)^2 : 
 x,y\text{ are }(n,\beta)\text{-qualified, }l(x,y)\ge  (1-\beta)n +\eta\bigg\}\bigg|\bigg]\\
\le &\bigg|D\bigg(z, \frac{r_{n,(1-\beta)n}}{2}\bigg)\bigg|^2(1+o(1))n^{c_3}q_n^2 
\sum_{l=(1-\beta)n +\eta }^nn^{-6\eta+7-6(l-n+\beta n-\eta+1)}e^{c_4(l-n+\beta n-\eta+1) \log \log n}\\
\le &o(1)
E\bigg[\bigg|\bigg\{x\in D\bigg(z, \frac{r_{n,(1-\beta) n}}{2}\bigg) :  x\text{ is }(n,\beta)\text{-qualified} \bigg\}\bigg|\bigg]^2,
\end{align*}
and 
 \begin{align*}
&E\bigg[\bigg|\bigg\{(x,y)\in D\bigg(z, \frac{r_{n,(1-\beta)n}}{2}\bigg)^2 :  x,y\text{ are }(n,\beta)\text{-qualified, }l(x,y)\le (1-\beta)n +\eta-1\bigg\}\bigg|\bigg]\\
\le &(1+o(1))\bigg|D\bigg(z, \frac{r_{n,(1-\beta)n}}{2}\bigg)\bigg|^2q_n^2 .
\end{align*}
Then, 
 \begin{align*}
&E\bigg[\bigg|\bigg\{x\in D\bigg(z, \frac{r_{n,(1-\beta)n}}{2}\bigg) :  x\text{ is }(n,\beta)\text{-qualified} \bigg\}\bigg|^2\bigg]\\
\le&
(1+o(1))E\bigg[\bigg|\bigg\{x\in D\bigg(z, \frac{r_{n,(1-\beta)n}}{2}\bigg) :  x\text{ is }(n,\beta)\text{-qualified} \bigg\}\bigg|\bigg]^2
\end{align*}
uniformly in $z\in \mathcal{A}_n$. In addition, Proposition \ref{r1**} implies that, for all sufficiently large $n\in \mathbb{N}$,
 \begin{align*}
E\bigg[\bigg|\bigg\{x\in D\bigg(z, \frac{r_{n,(1-\beta)n}}{2}\bigg) :  x\text{ is }(n,\beta)\text{-qualified} 
\bigg\}\bigg|\bigg]
\ge K_n^{2\beta- 2\alpha(1-\gamma(1-\beta))^2/\beta-\delta}.
\end{align*}
Then, by the Paley--Zygmund inequality, 
we have that, as $n\to \infty$,
 \begin{align*}
P\bigg(\bigg|\bigg\{x\in D\bigg(z, \frac{r_{n,(1-\beta) n}}{2}\bigg) :  x\text{ is }(n,\beta)\text{-qualified} 
\bigg\}\bigg|
\ge K_n^{2\beta- 2\alpha(1-\gamma(1-\beta))^2/\beta-2\delta}\bigg)\to 1.
\end{align*}
Therefore, by Proposition \ref{r1*}, as $n\to \infty$,  
 \begin{align*}
P\bigg(&\bigg|\bigg\{x\in D\bigg(z, \frac{r_{n,(1-\beta) n}}{2}\bigg) :  K(\hat{\mathcal{R}}^z_{(1-\beta)n,\tilde{f}((1-\beta) n)-(1-\beta) n },x)\ge 
\bigg(\frac{4 \alpha}{\pi}-\frac{1}{\log \log n} \bigg)(\log K_n)^2 \bigg\}\bigg|\\
&\ge K_n^{2\beta- 2\alpha(1-\gamma(1-\beta))^2/\beta-2\delta}\bigg)\to 1,
\end{align*}
and hence, 
the desired result holds for all sufficiently large $n\in \mathbb{N}$ with $1/\log_2 n <\delta$.
\end{proof}

%%%%%%%%%%%%%%%%%%%%%%%%%%%%%%%%%%%%%%%%%%%%%%%%%%%%%%%%%%%%%

Finally, to complete the proof of Lemma \ref{r2}, 
we present the following lemma. 
We say that a point $z \in \mathcal{A}_n$ is $(n,\beta)$-pre-successful  
if $|\hat{N}_{n,k}^z - \tilde{f}(k)|\le k$ 
for all $3\le k\le n-\beta n$. 
\begin{lem}\label{r*2*}
There exists $\delta_n> 0$ with $\lim_{n \to \infty} \delta_n=0$ such that
\begin{align}\label{qe1***}
\overline{q}_n 
:=\inf_{x\in \mathcal{A}_n } 
P(x \text{ is } (n,\beta)\text{-pre-successful})
\ge r_{(1-\beta )n}^{-(2\gamma^2\alpha+\delta_n)}.
\end{align}
Then, there exists some $c>0$ such that, for all $x \neq y \in \mathcal{A}_n$,
\begin{align}\label{qe2***}
P(x \text{ and }y \text{ are } (n,\beta)\text{-pre-successful})
\le c\overline{q}_n^2 r_{l(x,y)}^{2\gamma^2\alpha+\delta_{l(x,y)}}.
\end{align}
%where $l'(x,y)=\min \{j\ge1 ; D(x,r_{n,j})\cap D(y,r_{n,j})=\emptyset \}\wedge \beta n$ and $l'(x,y)\le \beta n$ when $x \neq y \in  {\cal A}_n$.
\end{lem}
\begin{proof}
As mentioned in \cite{Dembo1},  (\ref{qe1***}) and (\ref{qe2***}) 
are essentially established in \cite[Lemma $3.2$]{rosen}, 
although \cite{rosen} considered the case in which $\beta=0$ and 
used a different $r_{n,k}$ from ours. 
We provide an overview of the proof. 
First, we show (\ref{qe1***}).  
Note that, from (\ref{g2}), for any $1\le k\le n-\beta n-1$ and $x \in \partial D(z, r_{n,k})$,
\begin{align*}
P^x(T_{\partial D(z,r_{n,k+1})}<T_{\partial D(z,r_{n,k-1})})
=\frac{1}{2}+O(n^{-8}).
\end{align*}
Using \cite[Lemma $4.1$]{rosen} 
and the same argument as in the proof of Proposition \ref{r1**},  
\begin{align}\notag
&\inf_{x\in \mathcal{A}_n } 
P(x \text{ is } (n,\beta)\text{-pre-successful})\\
\notag
\ge &(1+o(1))
\sum_{m_k: |m_k-\tilde{f}(k)|\le k}
\prod_{k=3}^{n-\beta n-1}
\begin{pmatrix}
m_{k+1}+m_k-1\\
\notag
m_{k+1}
\end{pmatrix}
\bigg(\frac{1}{2} \bigg)^{m_{k+1}+m_k}
\\
\label{df}
\ge &r_{(1-\beta )n}^{-(2\gamma^2\alpha+\delta_n)},
\end{align}
and hence, we obtain (\ref{qe1***}). 
Next, we prove (\ref{qe2***}). 
As a result of applying Lemma \ref{ghj}, 
\begin{align*}
&P(x \text{ and }y \text{ are } (n,\beta)\text{-pre-successful})\\
\le &\sum_{m_l: |m_l-\tilde{f}(l)|\le l}
P(|\hat{N}_{n,k}^x - \tilde{f}(k)|\le k \text{ for all }  k=3,\ldots,l-3,l,\ldots,n; 
\hat{N}_{n,l}^y = m_l;\\
&|\hat{N}_{n,k}^y - \tilde{f}(k)|\le k  \text{ for all }  k=l+1,\ldots,n-\beta n )\\
\le &CP(|\hat{N}_{n,k}^x - \tilde{f}(k)|\le k  \text{ for all }  k=3,\ldots,l-3,l,\ldots,n)\\
&\sum_{m_l:|m_l-\tilde{f}(l)|\le l}
P(|\hat{N}_{n,k}^y - \tilde{f}(k)|\le k \text{ for all }  k=l+1,\ldots,n-\beta n 
|\hat{N}_{n,l}^y = m_l).
\end{align*}
By the same argument as (\ref{df}) (see also \cite[(4.29)]{rosen}), 
 \begin{align}\label{ree}
 \sum_{m_l: |m_l-\tilde{f}(l)|\le l}
P(|\hat{N}_{n,k}^y - \tilde{f}(k)|\le k \text{ for all }  k=l+1,\ldots,n-\beta n
 |\hat{N}_{n,l}^y = m_l)
\le c\overline{q}_n r_{l(x,y)}^{2\gamma^2\alpha+\delta_{l(x,y)}}.
\end{align}
In addition, by Lemma \ref{ghj}, 
\begin{align*}
&P(|\hat{N}_{n,k}^x - \tilde{f}(k)|\le k  \text{ for all } k=3,\ldots,l-3,l,\ldots,n)\\
\le &P(|\hat{N}_{n,k}^x - \tilde{f}(k)|\le k  \text{ for all }  k=3,\ldots,l-3,l,\ldots,n)\\
\times & \sum_{m_l: |m_l-\tilde{f}(l)|\le l}
P(|\hat{N}_{n,k}^x - \tilde{f}(k)|\le k  \text{ for all }  k=l+1,\ldots,n
|\hat{N}_{n,l}^x = m_l).
\end{align*}
By the same argument as for (\ref{df}) and (\ref{ree}), 
\begin{align*}
P(|\hat{N}_{n,k}^x - \tilde{f}(k)|\le k  \text{ for all }  k=3,\ldots,l-3,l,\ldots,n)
\le c\overline{q}_n r_{l(x,y)}^{\delta_{l(x,y)}}.
\end{align*}
With (\ref{ree}), we have (\ref{qe2***}). 
\end{proof}

\begin{proof}[Proof of Lemma \ref{r2}]
As an $(n,\beta)$-successful point is also an $(n,\beta)$-pre-successful point, 
(\ref{qe2***}) leads to (\ref{qe2}). 
To show (\ref{qe2}), we refer to the proof of  \cite[Lemma $2.1$]{Dembo1} or  \cite[Lemma $10.1$]{Dembo2}. 
% Lemma \ref{r2} is obtained by the same argument as Lemma $2.1$ in \cite{Dembo1}. 
Substitute $n$, $1-\beta$, $l(x,y)$, and $2\gamma^2\alpha$ 
for $m$, $\beta$, $k(x,y)$, and $a$ in \cite[Lemma $2.1$]{Dembo1}, respectively, 
and consider our event $H_{(1-\beta) n}^x$ and 
$ \sup_{x\in \mathcal{A}_n} P(H^x_{(1-\beta)n, \tilde{f}((1-\beta)n)-(1-\beta)n})$ in (\ref{ju}) 
instead of $H_{\beta m}^x$ and $\xi_m$ in \cite{Dembo1}. 
Then, using Lemma \ref{r1} instead of  \cite[$(2.13)$]{Dembo1}, 
we have that 
\begin{align}\label{qe1****}
P(x \text{ is } (n,\beta)\text{-successful})
\ge 
(1+o(1))P(x \text{ is } (n,\beta)\text{-pre-successful}).
\end{align} 
uniformly in $x\in \mathcal{A}_n $.
An overview of the proof is as follows. 
By the same argument as in the proof of Lemma \ref{ghj}, 
\begin{align*}
&P(|\hat{N}_{n,k}^x - \tilde{f}(k)|\le k  \text{ for all }  k =3,4,\ldots, n-\beta n)\\
=&
\sum_{m_l:|m_l - \tilde{f}(l)|\le l}
P(|\hat{N}_{n,k}^x - \tilde{f}(k)|\le k \text{ for all }  k =3,4,\ldots, n-\beta n-1;\\
&\quad \quad\quad \quad \quad
\hat{N}_{n,n-\beta n}^x=m_l;
 H^x_{(1-\beta)n, \tilde{f}((1-\beta)n)-(1-\beta)n})\\
=&(1+o(1))(1-\sup_{x\in \mathcal{A}_n} 
P(H^x_{(1-\beta)n, \tilde{f}((1-\beta)n)-(1-\beta)n}))
P(|\hat{N}_{n,k}^x - \tilde{f}(k)|\le k  \text{ for all }  k =3,4,\ldots, n-\beta n)\\
=&
(1+o(1))P(|\hat{N}_{n,k}^x - \tilde{f}(k)|\le k  \text{ for all }  k =3,4,\ldots, n-\beta n).
\end{align*}
Then, from Lemma \ref{r1}, 
\begin{align*}
&P(x \text{ is } (n,\beta)\text{-successful})
\ge 
(1+o(1))P(|\hat{N}_{n,k}^x - \tilde{f}(k)|\le k  \text{ for all }   k =3,4,\ldots, n-\beta n),
\end{align*}
and hence, we obtain (\ref{qe1****}). 
%(Note that ($2.5$) in \cite{Dembo1} is corresponded to (\ref{ffh}). ) 
Therefore, with the aid of (\ref{qe1***}) and (\ref{qe1****}), we have (\ref{qe1}). 
\end{proof}

\section*{Acknowledgments. }
The author would like to thank Professor K. Kuwada for suggesting the problem and stimulating discussions and  Professor K. Uchiyama for valuable comments that led to a significant improvement over earlier work. 
In addition,  the author would like to thank H. Murakami, K. Suzuki, and C. Nakamura for checking English.

%%%%%%%%%%%%%%%%%%%%%%%%%%%%%%%%%%%%%%%%%%%%%%%%%%%%%%%%%%%%%%%%%%

\end{document}